\documentclass[12pt,twoside]{amsart}
\usepackage{amssymb}
\usepackage{a4wide}
\usepackage[latin1]{inputenc}
\usepackage[T1]{fontenc}
\usepackage{times}

\def\hpq0{h^{p,q}_{\leq 0}}
\def\Hpq0{\H_{\leq 0}^{p,q}}

\def\ds{\partial}

\def\C{{\mathbb C}}
\def\Cn{\C^n}

\def\H{{\mathcal H}}

\def\O{{\mathcal O}}

\def\cdelta{C_{\delta} }

\def\be{\begin{equation}}
\def\ee{\end{equation}}

\def\bl{\begin{lma}}
\def\el{\end{lma}}

\def\XXint#1#2#3{{\setbox0=\hbox{$#1{#2#3}{\int}$}
\vcenter{\hbox{$#2#3$}}\kern-.5\wd0}}

\newtheorem{thm}{Theorem}[section]
\newtheorem{lma}[thm]{Lemma}
\newtheorem{cor}[thm]{Corollary}
\newtheorem{prop}[thm]{Proposition}

\theoremstyle{definition}

\newtheorem{df}{Definition}

\theoremstyle{remark}

\newtheorem{preremark}{Remark}
\newtheorem{preex}{Example}

\newenvironment{remark}{\begin{preremark}}{\qed\end{preremark}}
\newenvironment{ex}{\begin{preex}}{\qed\end{preex}}

\numberwithin{equation}{section}

\begin{document}
\title[]
{A new generalization of the Lelong number.}

\author{Aron Lagerberg}

\address{A Lagerberg :Department of Mathematics\\Chalmers University
  of Technology 
  and the University of G\"oteborg\\S-412 96 G\"OTEBORG\\SWEDEN,\\}
\begin{abstract}
We will introduce a quantity which measures the singularity of
a plurisubharmonic function $\varphi$ relative to another plurisubharmonic function $\psi$, at a point $a$.
We denote this quantity by $ \nu_{a,\psi}(\varphi)$.
It can be seen as a generalization of the classical Lelong number in a natural way: 
if $\psi=(n-1)\log| \cdot - a|$ where $n$ is the dimension of the set where $\varphi$ is defined, then  $\nu_{a,\psi}(\varphi)$ coincides with the classical Lelong number of $\varphi$ at the point $a$.
The main theorem of this article says that the upper level sets of our generalized Lelong number, i.e. the 
sets of the form $ \{z: \nu_{z,\psi}(\varphi) \geq c \}$ where $c>0$,
are in fact analytic sets, provided that the \textit{weight} $\psi$ satisfies some additional conditions.
\end{abstract}

\email{ aronl@chalmers.se}
\maketitle
\tableofcontents

\section{Introduction} \label{introduction}
In what follows, we let $\Omega$ denote an open subset of $\C^n$, $ \varphi$  a plurisubharmonic function in $\Omega$,
and $\psi$ a plurisubharmonic function in $\Cn$. When we are dealing with constants, we often let the same symbol
denote different values when the explicit value does not concern us.
The object of this paper is to introduce a generalization of the classical Lelong number:
The quantity we will consider depends on two plurisubharmonic functions $\varphi, \psi$
and it will be a measurement of the singularity of $\varphi$ relative to $\psi$.
Moreover, if we let $\psi(z) = (n-1)\log|z-a|$ we get back the classical Lelong number of $\varphi$ at the point $a$.
The main theorem of this paper (Theorem \ref{analyticity_thm}) tells us that this generalized Lelong number satisfies a semi-continuity property of the same
type as the classical Lelong number does, namely, its super level-sets define analytic varieties. 
Also, we investigate what further properties this quantity satisfies.
The paper is organized as follows: in this introduction we define the generalized Lelong number and discuss the motivation behind it.
In section (\ref{properties_and__examples}) we explore some basic properties and examples of the generalized Lelong number, obtaining
as corollaries classical results concerning the classical Lelong number. Section \ref{analyticity_section} concerns the theorem stating
that the upper level-sets of the generalized Lelong number defines an analytic set. In section \ref{approx_section} we prove a theorem
due to Demailly, which states that one can approximate plurisubharmonic functions well with Bergman functions with respect to a certain weight.
Using a different weight we obtain a slightly better estimate. In section \ref{kiselmans_section} we relate our generalized
Lelong number to another generalization due to Kisleman (cf. \cite{Kiselman}).
\\ \\
Let us begin by recalling some relevant definitions.
For $r>0$ define
\be \label{leldef} 
 \nu ( \varphi,a,r) := \frac{\sup_{|z-a|=r} \varphi(z)}{\log r}.
 \ee
The function in the nominator can actually be seen to be a convex function of $\log r$ (cf. \cite{Lelong}). Furthermore, the fraction is increasing in $r$, and so the limit as $r$ tend to 0 exists:

\begin{df}
The (classical) Lelong number of $\varphi$ at $a\in\Omega$ is defined as
\be \label{lel_def_2}
\nu (\varphi, a) = \lim_{r \rightarrow 0} \nu ( \varphi,a,r).
\ee
\end{df}

As can be seen from the definition, the Lelong number compares the behaviour of $\varphi$ to that of $\log(|z-a|)$, as $z \rightarrow a$.
In fact (cf. \cite{Kiselman}), the following is true:
if $\nu (\varphi, a)=\tau$
then, near the point $a$, 
\be \label{liminfcond}
\varphi(z) \leq \tau \log|z-a| + O(1),
\ee
 and $\tau$ is the best constant possible. 
 %Many results
%concerning Lelong numbers have been discovered. Some of them are easy consequences of the fact that they are defined as limits,
%and some of them resides deeper within the definition. 
Two other ways to represent the classical Lelong number are given by the following equalities:
$$\nu (\varphi, a) = 
\liminf_{z \rightarrow a} \frac{ \varphi(z)}{\log|z-a|}
$$
and
$$
\nu (\varphi, a)
= \lim_{r \rightarrow 0} \int_{B(a,r)} (d d^c \varphi(z)) \wedge (dd^c \log|z-a|^2)^{n-1}.
$$
The first of these equalities is a simple consequence of (\ref{leldef}) (cf.\cite{Rashkovskii}), while the other
follows from Stokes' theorem (cf.\cite{Demailly}). Two generalizations of the Lelong number, due to Rashkovskii 
and Demailly 
respectively, come from exchanging $\log|z-a|$ for a different plurisubharmonic function $\psi$ in the charactarizations of the Lelong number above (cf \cite{Rashkovskii} and \cite{Demailly} respectively).
To that effect, the relative type of $\varphi$ with respect to a function $\psi$ is given by
\be \label{relativetypedef}
\sigma_a (\varphi, \psi) = 
\liminf_{z \rightarrow a} \frac{ \varphi(z)}{\psi(z-a)},
\ee
and Demailly's generalized Lelong number of $\varphi$ with respect to $\psi$ is given by
\be
\nu_{Demailly} (\varphi, \psi)
= \lim_{r \rightarrow 0} \int_{\psi < \log r} (d d^c \varphi) \wedge (dd^c \psi)^{n-1}.
\ee
Of course, in each of the definitions above, $\psi$ needs to satisfy some regularity conditions; for our discussion
it suffices to assume that $ e^\psi $ is continuous and $\{ \psi = - \infty \} = \{ 0 \}$.
\\
The inspiration for this article comes from an observation made by Berndtsson,
who in the article  \cite{Berndtsson} relates the classical Lelong number to the convergence of a certain integral:

\begin{thm} \label{berndtssons_sats}
For $a \in \Omega$,
$$ \nu_{a} (\varphi) \geq 1 \Longleftrightarrow \int_a e^{-2 \varphi(\zeta) - 2(n-1)\log |\zeta - a|} d \lambda(\zeta) = + \infty.$$
\end{thm}
We will giva a simplified proof of this Theorem in section \ref{properties_and__examples}.
Using this result, assuming $\nu_{a} (\varphi) = \tau$ so that $\nu_{a} (\varphi/ \tau)=1$, we see that 
$$ \int_a e^{-2 \frac{\varphi(\zeta)}{s} - 2(n-1)\log |\zeta - a|} d \lambda(\zeta) = 
\int_a e^{-2 \frac{\tau \varphi(\zeta)}{s \tau} - 2(n-1)\log |\zeta - a|} d \lambda(\zeta)$$
is infinite iff $s \leq \tau$.
Thus Theorem \ref{berndtssons_sats} implies that the classical Lelong number coincides with the number given by 
$$\inf{ \{ s>0:  \zeta \mapsto e^{ - \frac{2 \varphi (\zeta)}{s} - 2 (n-1) \log \left| \zeta - a \right|   } \in L_{Loc}^1(a) \} }.$$
In \cite{Berndtsson} the following generalization of the classical Lelong number is indicated:
%%%%%%%%%%%%%%%%%%%%%%%%%%%%%%%%
\begin{df}
The \textit{generalized Lelong number of} $\varphi$ at $a\in \Omega$ with respect to a plurisubharmonic function $\psi$, is defined as 
$$
\nu_{a,\psi} (\varphi) = \inf{ \{ s>0:  \zeta \mapsto e^{ - \frac{2 \varphi (\zeta)}{s} - 2 \psi( \zeta - a )  } \in L_{Loc}^1(a) \} }. 
$$
\end{df}

Obviously, some condition regarding the integrability of $e^{-2\psi}$ is needed for the definition to provide us with something of interest;
for our purpose, it is sufficient to assume that 
\be \label{integrabilitycondition2}
e^{-2(1+\tau)\psi} \in L_{Loc}^1(a),
\ee  for some $\tau>0$.
We single out the following special case of the generalized Lelong number:
\begin{df}
For $t \in [0,n)$ we define
$$
\nu_{a,t} (\varphi) = \inf{ \{ s>0:  \zeta \mapsto e^{ - \frac{2 \varphi (\zeta)}{s} - 2 t \log \left| \zeta - a \right|  } \in L_{Loc}^1(a) \} }, 
$$
that is,  $\nu_{a,t} (\varphi):= \nu_{a,\psi} (\varphi)$, with $\psi=t \log \left| \zeta - a \right|$. 
\end{df}

%%%%%%%%%%%%%%%%%%%%%%%%%%%%%%%%
Theorem \ref{berndtssons_sats} shows that $\nu_{a,n-1}$ equals the classical Lelong number, which we sometimes will denote by just $\nu_{a} $. 
For $t=0$, $\nu_{a,t}=\nu_{a,0}$ equals another well known quantity, the so called \textit{ integrability index of $\varphi$ at $a$}.
Thus $\nu_{a,t }$ can be regarded as a family of numbers which interpolate between the classical Lelong number and the integrability index of $\varphi$,
as $t$ ranges between $0$ and $n-1$.
One should put this in context with the following important inequality, due to Skoda (cf. \cite{Skoda}):

\begin{thm}(Skoda's inequality) For $\varphi \in PSH(\Omega)$,
\be \label{skoda_ineq}
 \nu_{z,0}(\varphi) \leq \nu_{z,n-1}(\varphi) \leq n \nu_{z,0}(\varphi).
\ee
\end{thm}
Later, we will prove the following generalization of Skoda's inequality: 
$$\nu_{z,0}(\varphi) \leq \nu_{z,n-1}(\varphi) \leq  (n-t)\nu_{z,t}(\varphi)     \leq  n \nu_{z,0}(\varphi).$$

\begin{remark}
Observe that when $n=1$ the Lelong number and integrability index of a function coincide. This 
follows from, for instance, Theorem \ref{berndtssons_sats}.
\end{remark}

A well known important result concerning classical Lelong numbers, due to Siu (cf. \cite{Siu}), is the following:
\begin{thm}
The sets $\{ z \in \Omega : \nu_z( \varphi ) \geq \tau  \} $ are analytic subsets of $\Omega$, for $\tau>0$.
\end{thm}
In fact both the relative type and Demailly's generalized Lelong number defined above, enjoy similar
analyticity properties, provided that $e^{2 \psi}$ is Hölder continuous.
A natural question arises: does an equivalent statement to Siu's analyticity theorem hold for Berndtsson's generalized Lelong number? More precisely,
are the sets $$ \{ z \in \Omega : \nu_{z,\psi}( \varphi) \geq \tau  \}  $$
analytic for $\tau>0$? In the case where the weight $e^{2 \psi}$ is Hölder continuous the affirmative answer is the content of theorem $\ref{analyticity_thm}$.
The main idea of the proof is due to Kiselman (cf.(\cite{Kiselman})) and consists in his technique of "attenuating the singularities of $\varphi$". However, this is here done in a different
manner than in \cite{Kiselman}, following results from \cite{Berndtsson}.
Attenuating the singularities means that we construct a plurisubharmonic function $\Psi$ satisfying the following properties: if the generalized Lelong number
of $\varphi$ is large then its classical Lelong number is positive, and
if the generalized Lelong number of $\varphi$ is small then its classical Lelong number vanishes. 
Using this function we can then realize the set $\{ z \in \Omega : \nu_{z,\psi}( \varphi ) \geq \tau  \}$ as an intersection of analytic sets, which by basic properties of analytic sets is analytic. 
\\

\textbf{Acknowledgements:}{
I would like thank my advisor Bo Berndtsson for introducing me to the topic of this article, for his great knowledge and inspiration, and for his continuous support.
}
\section{Properties and examples} \label{properties_and__examples}

We begin with listing some properties which the generalized Lelong number satisfies.
\bl
Let $\varphi, \varphi' \in PSH( \Omega)$, and assume that $\psi$ satisfies (\ref{integrabilitycondition2}) . Then 
\begin{enumerate}
	\item  For $c>0$ , $\nu_{a,\psi}(c \varphi ) = c \nu_{a,\psi}(\varphi) $.
	\item  {If $\varphi \leq \varphi'$ on some neighbourhood of $a \in \Omega$, then $\nu_{a,\psi}( \varphi )  \geq \nu_{a,\psi}(\varphi' )$.}
	\item { $\nu_{a,\psi}(\varphi + \varphi') \leq \nu_{a,\psi}(\varphi) +\nu_{a,\psi}(\varphi') $.}
	\item { $\nu_{a,\psi}(\max(\varphi, \varphi')) \geq \min(\nu_{a,\psi}({\varphi}), \nu_{a,\psi}(\varphi') )$.}
	\item { Assume that $\nu_{a,0}( \varphi ) \leq \sigma_a(\varphi, \psi):=\sigma$, 
	where $\sigma_a(,)$ denotes the relative type as defined by (\ref{relativetypedef}).
	Then, $$\nu_{a,\psi}(\varphi) \leq \frac{\nu_{a,0}(\varphi)}{1-\frac{\nu_{a,0}(\varphi)}{\sigma}}.$$
	If $\nu_{a,0}( \varphi ) > \sigma_a(\varphi, \psi)$
	then $\nu_{a, \psi}(\varphi)=0.$
	}
	\end{enumerate}
\el

\begin{proof}
The first properties, (1) and (2), are immediate consequences of the definition. \\
(3): Let $s_0>\nu_{a,\psi}(\varphi)$ and $s'_0 > \nu_{a,\psi}(\varphi')$. 
Put$$p=\frac{s_0+s'_0}{s_0}, q = \frac{s_0+s'_0}{s'_0} $$
so that $ \frac{1}{p} + \frac{1}{q}=1$. Hölder's inequality with respect to the finite measure $ e^{-2\psi} d \lambda $ on
some neighbourhood $U$ of $a$, gives us:
$$
\int_{U} e^{ - 2 \frac{\varphi + \varphi' }{s_0  + s'_0} - 2 \psi ( \zeta - a )  } \leq
\left(   \int_{U} e^{ - 2 \frac{\varphi}{s_0} - 2 \psi ( \zeta - a )  } \right)^{1/p}
\left(   \int_{U} e^{ - 2 \frac{\varphi'}{s'_0} - 2 \psi ( \zeta - a )  } \right)^{1/q} < +\infty
$$
due to our choice of $s_0$ and $s'_0$. Thus $\nu_{a,\psi}(\varphi + \varphi') \leq s_0 + s'_0$. Since
we can choose $s_0$ and $s'_0$ arbitrarily close to $\nu_{a,\psi}(\varphi)$ and $\nu_{a,\psi}(\varphi')$ respectively, we are done. \\
(4): Let $s_0<\nu_{a,\psi}(\varphi)$ and $s'_0 < \nu_{a,\psi}(\varphi')$, and let $s=\min(s_0,s'_0)$. Then
$$
\int_{U} e^{ - 2 \frac{\max(\varphi,\varphi') }{s} - 2 \psi ( \zeta - a )  } = \infty.
$$
Thus  $\nu_{a,\psi}(\max(\varphi, \varphi')) \geq s$. The statement follows.\\
(5):
In \cite{Rashkovskii} it is deduced that $$ \varphi(z) \leq \sigma_a(\varphi, \psi) \psi(z) + O(1),$$
as $z \rightarrow a$ ( cf. (\ref{liminfcond})).
Thus, if we choose $r>0$ small enough, and $ \sigma \leq \nu_{a,0}( \varphi )$,
$$
\int_{B(a,r)} e^{ - 2 \frac{\varphi(\zeta)}{s} - 2 \psi ( \zeta - a )  } 
\leq
C \int_{B(a,r)} e^{ - 2 \frac{\varphi(\zeta)}{s} - 2\frac{\varphi(\zeta)}{\sigma} }
$$
which is finite if (remember that $\nu_{a,0}$ denotes the integrability index)
$$\frac{1}{s} + \frac{1}{\sigma} < \frac{1}{\nu_{a,0}(\varphi)} 
\Leftrightarrow 
s > \frac{\nu_{a,0}(\varphi)}{1-\frac{\nu_{a,0}(\varphi)}{\sigma}} .
$$
Thus we obtain:
$$
\nu(\varphi,\psi) \leq \frac{\nu_{a,0}(\varphi)}{1-\frac{\nu_{a,0}(\varphi)}{ \sigma}}.
$$
On the other hand, it is evident that if $\sigma > \nu_{a,0}( \varphi )$ then the integral above will always
be infinite, whence $\nu(\varphi,\psi)=0.$
\end{proof}

We proceed by listing properties the special case $\nu_{z,t \psi}$ satisfies:

\bl \label{concavity_lemma}
For $\varphi, \psi$ plurisubharmonic we have:
\begin{enumerate}
\item{ 
The function $$ t \mapsto \frac{1}{\nu_{z,t \psi}(\varphi)} $$ is concave
while $$ t \mapsto {\nu_{z,t \psi}(\varphi)} $$ is convex.
}
\item{
If $\psi $ is such that $e^{-2n\psi} \notin L_{Loc}^1(0)$ then 
the function
$$ t \mapsto (n-t) \nu_{z,t \psi}(\varphi)$$
is decreasing.
}
\item{
The following inequalities hold:
\be \label{generalized_skoda_inequality}
 \nu_{z,0}(\varphi) \leq \nu_{z,n-1}(\varphi) \leq (n-t) \nu_{z,t} (\varphi) \leq n \nu_{z,0}(\varphi).
\ee
}
\end{enumerate}
\el
\begin{proof}
(1):
Assume $z=0$ and put $$f(t) = 1/ \nu_{0,t}( \varphi) = \sup\{s>0 : e^{-s 2\varphi -2t \psi} \in L_{Loc}^1(0) \}.$$ 
By the definition of concavity, we need to show that
for every $a,b \in [0,n)$ and $\lambda \in (0,1)$ the inequality
$$ f(\lambda a + (1 - \lambda)b) \geq \lambda f(a) + (1 -\lambda)f(b)$$
holds.
Applying Hölder's inequality once again, with $p=1/\lambda, q=1/(1-\lambda)$, we see that
$$
\int_0 e^{-2 (\lambda f(a) + (1-\lambda) f(b))  \varphi -2(\lambda a + (1-\lambda) b) \psi}  \leq
\big( \int_0 e^{-2 f(a) \varphi -2 a  \psi} \big)^{\lambda}
\big(\int_0 e^{-2 f(b) \varphi -2 b \psi} \big)^{1-\lambda},
$$
which implies that 
$ f(\lambda a + (1 - \lambda)b) \geq  \lambda f(a) + (1-\lambda) f(b)$ .
Thus $f$ is a concave function of $t$. 
The exact same calculations with $f(t) =  \nu_{0,t \psi}( \varphi)$ give convexity of
$t \mapsto {\nu_{z,t \psi}(\varphi)}$. Note however that this statement is weaker than saying that 
$t \mapsto \frac{1}{\nu_{z,t \psi}(\varphi)}$ is concave. \\
(2): One can show that if $f(t) \geq 0$ is a concave function with $f(0)=0$, then $t \mapsto f/t$ is decreasing.
Since $ t \mapsto 1/\nu_{0,(n-t)\psi}(\varphi)$ is concave by property ($1$) and is equal to $0$ for $t=0$ by the condition on $\psi$, 
we see that $$ \frac{1}{t\nu_{0,(n-t)\psi}(\varphi)}$$
is decreasing. This implies that $t \mapsto (n-t) \nu_{0,t \psi}(\varphi) $ is a decreasing function. \\
(3): If we accept Skoda's inequality ($\ref{skoda_ineq}$), the only new information is the inequality
$$ \nu_{0,n-1}(\varphi) \leq (n-t)\nu_{0,t}(\varphi)  \leq n \nu_{0,0}(\varphi) $$
which follows immediately from property ($2$) with $\psi=\log| \cdot |$, that is, the fact that  $t \mapsto (n-t)\nu_{0,t}(\varphi)$ is decreasing in $t$.
\end{proof}

\begin{remark}
The proof of property (1) in Lemma \ref{concavity_lemma} can easily be adapted to show that something
stronger holds:
the function $$ \psi \mapsto \frac{1}{\nu_{z,\psi}(\varphi)}$$ is concave on the set of plurisubharmonic functions $\psi$.

\end{remark}

We proceed by calculating two special cases of the generalized Lelong number, which will give us some insight to what it measures.
\begin{ex} \label{example1}
We calculate $ \nu_{0,t}(\varphi)$ where $ \varphi(z_1,..z_n)=\frac{1}{2}\log (z_1\bar{z_1}+...z_k \bar{z_k})$ where $k$ lies between 1 and $n$. 
Thus we want to decide for which $s>0$ the following integral goes from being finite into being infinite:
$$ \int_{\Delta} \frac{d \lambda}{|z_1\bar{z_1}+...z_k \bar{z_k}|^{1/s} |z|^{2t}},$$
where $\Delta$ is some arbitrarily small polydisc containing the origin.
In this integral we put $z'' = (z_{k+1},..z_n)$ and introduce polar coordinates with respect
to $z'=(z_1,..z_k)$ to obtain that it is equal to
\begin{equation} \label{specint1} 
C \int_{\Delta ''} \int_0^1 \frac{R^{2k-1-2/s} }{|R^2 + |z''|^2 |^{t}}dR d \lambda(z''),
\end{equation}
were $C$ is some contant depending only on the dimension.
%To see for which $s$ ($\ref{specint1}$) is finite, we divide and truncate the inner integral into two parts  
% by dividing the domain of integration into $[\epsilon,|z''|] $ and $[|z''|,1]$ for $\epsilon > 0$. The first part  
% $$I = \int_{\epsilon}^{|z''|} \frac{R^{2k-1-2/s} }{|R^2 + |z''|^2 |^{t}}dR $$ is comparable to $$\frac{|z''|^{2k-2/s} - {\epsilon}^{2k-2/s}}{2^t|z''|^{2t}(2k-2/s)}.$$
 %% \int_{\epsilon}^{|z''|} \frac{R^{2k-1-2/s} }{|z''|^{2t}}dR  =  \frac{|z''|^{2k-2/s} - {\epsilon}^{2k-2/s}}{|z''|^{2t}(2k-2/s)}
%%$$ I = \int_{\epsilon}^{|z''|} \frac{R^{2k-1-2/s} }{|R^2 + |z''|^2 |^{t}}dR \geq 
%% \int_{\epsilon}^{|z''|} \frac{R^{2k-1-2/s} }{ 2^t|z''|^{2t}}dR  =  \frac{|z''|^{2k-2/s} - {\epsilon}^{2k-2/s}}{2^t|z''|^{2t}(2k-2/s)}$$ and
%%$$ I = \int_{\epsilon}^{|z''|} \frac{R^{2k-1-2/s} }{|R^2 + |z''|^2 |^{t}}dR \leq 
%% \int_{\epsilon}^{|z''|} \frac{R^{2k-1-2/s} }{|z''|^{2t}}dR  =  \frac{|z''|^{2k-2/s} - {\epsilon}^{2k-2/s}}{|z''|^{2t}(2k-2/s)}$$ 
%The
%second part yields
%$$ II = \int_{|z''|}^{1} \frac{R^{2k-1-2/s} }{|R^2 + |z''|^2 |^{t}}dR \geq 
% 2^{-t}\int_{|z''|}^{1} R^{2k-1-2/s-2t} dR  = \frac{1 - {|z''|}^{2k-2t-2/s}}{2^t(2k-2/s-2t)}$$
%and
%$$ II  \leq 
%2^{-t} \int_{|z''|}^{1} \frac{R^{2k-1-2/s}}{|z''|^{2t}} dR  =  \frac{1 - {|z''|}^{2k-2/s}}{2^t|z''|^{2t} (2k-2/s)}.$$
%Analyzing these expressions we see that, when letting $\epsilon \rightarrow 0$, the integral in $ (\ref{specint1})$ is finite iff $ 2k -2/s> 0$ and $ 2k -2t-2/s >2k-2n$. In other words $$\nu_{0,t}(\log %(z_1\bar{z_1}+...z_k \bar{z_k})) = \max(\frac{1}{k},\frac{1}{n-t})  $$
This integral is easily seen to be finite if and only if $ 2k -2/s> 0$ and $ 2k -2t-2/s >2k-2n$. In other words $$\nu_{0,t}(\log (z_1\bar{z_1}+...z_k \bar{z_k})) = \max(\frac{1}{k},\frac{1}{n-t}) .$$

\end{ex} 

This example shows that when we look at sets of the type $\{ z_1=...=z_k=0 \}$ in $\mathbb{C}^n$ the generalized Lelong
number, as a function of $t$, thus senses the (co-)dimension of the set: it is constant, and equal to the integrability index of $\frac{1}{2}\log (z_1\bar{z_1}+...z_k \bar{z_k})$,
when $t$ is so small so that $n-t$ is larger than $k$ - the co-dimension of the set - and then grows linearly to $1$, which is the Lelong number of $\frac{1}{2}\log (z_1\bar{z_1}+...z_k \bar{z_k})$.

\begin{ex} \label{example2}
Next we compute $\nu_{0,t}(\varphi)$ for $ \varphi(z_1,..z_n)=\log (z_1^{\alpha_1} \cdots z_k^{\alpha_1})$ for $1\leq k \leq n$ and $(\alpha_1,..,\alpha_k) \in \ \mathbb{N}^k$, $\alpha_i \neq 0$.
Thus we want to study the following integrals behaviour with respect to $s$:
$$ \int_{\Delta} \frac{d \lambda}{|z_1^{\alpha_1} \cdots z_k^{\alpha_k}|^{2/s} |z|^{2t}}.$$
Using Fubini's thorem and changing to polar coordinates in each of the variables $z_1$ to $z_k$ we
obtain:
$$ \int_{ \Delta''} \int_{\mathbb{R}^k \cap U} \frac{r_1^{1-2 \alpha_1 /s} \cdots r_k^{1-2 \alpha_k /s } d r_1 \cdots d r_k }{ |r_1^2 + ... + r_k^2 + |z''|^2|^{t}}d \lambda(z'')$$
for some open set $U$ in $\mathbb{R}^k$.
We put $ N = \sum \alpha_i$ and in the inner integral we change to polar coordinates in $\mathbb{R}^k$ and obtain an integral of the same magnitude:
$$ \int_{ \Delta''} \int_{0}^r \frac{R^{2k-1-2 N/s} }{|R^2 + |z''|^2|^{t}}dR   d \lambda(z'') \cdot \int_{S^{k-1}} \omega_1^{1-2 \alpha_1 /s} \cdots \omega_k^{1-2 \alpha_k /s} d \sigma( \omega ),$$
for some $r>0$.
In the previous example saw that the transition into being infinite for the first integral  (with $s$ replaced by $s/N$) was obtained for $s=\max(\frac{N}{k},\frac{N}{n-t})$.
The second integral can be computed using gamma functions, and in fact equals
$$  \frac{\prod_{i=1}^k  \Gamma(\frac{ 2 - 2 \alpha_i /s}{2})}{\Gamma(\frac{ |\alpha| +k}{2})} ;$$
hence the integral diverges for $ 1-2 \alpha_i /s=-m$ for $m>0$ (if the other factors are non-zero). However, the condition on $s$ becomes $ s = \frac{\max{\alpha_i}}{1+m} \leq \frac{N}{k} $,
which will not give any contribution to $s=\max(\frac{N}{k},\frac{N}{n-t})$.
Putting it together we see that:
$$ \nu_{0,t}(\log (z_1^{\alpha_1} \cdots z_k^{\alpha_1})) = \max(\frac{\sum \alpha_i}{k},\frac{\sum \alpha_i}{n-t}).$$
\end{ex}
When we are looking on sets of the form $\{ z_1 \cdots z_k = 0\}$, which is the union of the $k$ coordinate planes
$\{ z_k=0 \}$ (corresponding to the function $\varphi=\log|z_1 \cdots z_k | $), the generalized
Lelong number thus senses how many coordinate planes the union is taken over. 
\begin{remark}
These two examples show us that the two leftmost inequalities in (\ref{generalized_skoda_inequality}) are sharp.
More precisely: using $\varphi$ from example \ref{example1}, we see that if $t=n-k$, $$ \nu_{0,n-1}(\varphi) \leq (n-t) \nu_{0,t}(\varphi) \Leftrightarrow 1 \leq \frac{n-t}{k} \Leftrightarrow 1 \leq 1.$$
However, if $\varphi$ is as in example \ref{example2}, with $t = n-k$ we see that
$$ (n-t)\nu_{0,t}(\varphi) \leq n \nu_{0,0}(\varphi) \Leftrightarrow \sum_i \alpha_i \leq \sum_i \alpha_i.$$

\end{remark}
%We get a similar interpretation in this case as in the previous example. One might
%guess then, as these two
%functions in a sense serves as the model case of plurisubharmonic functions,
%that this behaviour of the generalized Lelong number 
%(i.e. recognizing the dimension of the upper level set of the classical Lelong number)
%holds for every psh-function. However, the following example shows that this is not the case. 
%We have the same interpretation (modulo normalization) of the generlized Lelong number in this case as in the previous example. 
%One might guess then that this holds for general plurishubharmonic functions:
%the generalized Lelong number as a function of $t$, is constant and equal to the functions
%integrability index for $t$ between 0 and the co-dimension of the upper levelset of the
%Lelong number of the function to which the point we are looking at belongs, and then
%grows linearly to the Lelong number of the function. This is however not the case, 
%as the following example shows. 
%%%%%%
%%%%%%
We now use our generalized Lelong number to obtain a classical result,
namely, that $\nu_{a,t}$ is invariant under biholomorphic coordinate changes.
\begin{prop} \label{biholotheorem}
If $f:\Omega \rightarrow \Omega$ is biholomorphic, $f(0)=0$ and $\det f'(0) \neq 0$, then 
$$ \nu_{0,t} (\varphi \circ f ) = \nu_{0,t} (\varphi ).$$
\end{prop}
\begin{proof}

By a change of coordinates we get
$$ \int_0 e^{- \frac{2}{s}\varphi \circ f(\zeta) - 2t\log |\zeta|} d \lambda(\zeta) = 
\int_0 e^{- \frac{2}{s}\varphi (z) - 2t\log |f^{-1}(z)|} |\det f'(z)|^{-1} d \lambda(z).$$
Around the origin we have that $|f^{-1} (z) |$ is comparable to $|z|$ and thus the last integral is of the same magnitude as
$$\int_0 e^{- \frac{2}{s}\varphi (z) - 2t\log |z|} \frac{1}{|\det f'(z)|} d \lambda(z).$$
Since $C > \frac{1}{|\det f'(\cdot)|} > c>0$ in some neighbourhood of the origin, the first integral is infinite
iff  $$\int_0 e^{- \frac{2}{s}\varphi (\zeta) - 2t\log |\zeta|}  d \lambda(\zeta)=\infty.$$
In other words, $$ \nu_{0,t} (\varphi \circ f ) = \nu_{0,t} (\varphi ).$$ 
\end{proof}
Since for $t=n-1$ we get the classical Lelong number, we obtain as a corollary the theorem of Siu found in \cite{Siu}:
\begin{cor}
The classical Lelong number is invariant under biholomorphic changes of coordinates.
\end{cor}

\begin{ex}
Let $V \subset \Omega$ be a variety and pick a point $x \in V$ where $V$ is smooth. We can then find a neigbourhood $U$ of $x$ and $f_1,...,f_k \in \mathcal{O}(U)$
such that $$V \cap U = \{ f_1=...=f_k=0 \}.$$
Since $V$ was smooth at $x$, we can change coordinates via a function $g:U \rightarrow U$ such that in these new coordinates $$V \cap U = \{ z_1=...z_l=0 \}$$
for some $l \leq k$. This means that $f_i \circ g = z_i $ for $1 \leq i \leq l$ and $f_i \circ g = 0$ for $i \geq l$.
By proposition \ref{biholotheorem} we have that $$ \nu_{x,t} ( \sum_1^k |f_i|^2 ) = \nu_{x,t} ( \sum_1^l |z_i|^2 ), $$
and thus, by example \ref{example1} we see that $\nu_{x,t} ( \sum_1^k |f_i|^2 )$ senses the co-dimension of $V$ at $x$.
\end{ex}

When considering the generalized Lelong number, our next technical lemma shows that we can "move" parts of the singularity
from the plurisubharmonic function to the weight, provided the singularity is sufficiently large:
\begin{lma} \label{singularitylemma}
Let $ \delta > 0 $, and $\psi$  be a plurisubharmonic function such that $e^{-2(1+\tau)\psi} \in L_{Loc}^1(0)$, for some $\tau>0$. 
If $ \nu_{a,\psi} (\varphi) = 1 + \delta$, then with $0<\epsilon<\tau \delta$ we have that 
$$ \int_{a} e^{ - 2 \varphi (\zeta) - 2(1 - \epsilon ) \psi ( \zeta - a )  } = \infty.$$
\end{lma}

\begin{proof}
The hypothesis implies that for every neighbourhood $U$ of $a$, 
$$ 
\int_U e^{ - \frac{2 \varphi (\zeta)}{1 + \delta'} - 2 \psi ( \zeta - a )  }  = \infty,
$$  
when $\delta'<\delta$. 
The function $\zeta \mapsto  e^{-2(1-\epsilon)\psi (\zeta - a)}$ is locally integrable around $a$,
and we apply Hölder's inequality with respect to the measure 
$e^{-2(1-\epsilon)\psi (\zeta - a)} d\lambda (\zeta)$ on $U$, 
with $p = 1+\delta'$ and $q=\frac{1+\delta'}{\delta'}$, to obtain
\begin{eqnarray*} 
\int_U e^{ - \frac{2 \varphi (\zeta)}{1 + \delta'} -  2\psi (\zeta - a)}d \lambda (\zeta) 
=
\int_U e^{ - \frac{2 \varphi (\zeta)}{1 + \delta'}}  e^{-2(1-\epsilon)\psi (\zeta - a)} e^{- 2 \epsilon \psi (\zeta - a)} d\lambda (\zeta) 
\leq  \\
\leq \Big{(}\int_U e^{ - 2 \varphi (\zeta)}  e^{-2(1-\epsilon)\psi (\zeta - a)} d\lambda (\zeta)\Big{)}^{1/p} 
\Big{(} \int_{U}  e^{-2(\epsilon q + 1 - \epsilon)\psi (\zeta - a)} d\lambda (\zeta)  \Big{)}^{1/q}.
\end{eqnarray*}
Since the left hand side is infinite by hypothesis, 
and the second integral on the right hand side converges (after possibly shrinking $U$,
since $\epsilon q + 1 - \epsilon \leq  1 + \tau$, if $\epsilon < \delta' \tau$ ), we see that
$$ \Big{(}\int_U e^{ - 2 \varphi (\zeta) - 2(1-\epsilon) \psi (\zeta - a)}   d\lambda (\zeta)\Big{)}^{1/p} = \infty.$$
This implies the desired conclusion, since $\delta'$ can be choosen arbitrarily close to $\delta$.  
\end{proof}

%An important result in potential theory which we will have use of later on, is the following inequality due to Skoda (cf. \cite{Skoda}):
%\begin{thm}(Skoda's inequality) For $\varphi \in PSH(\Omega)$,
%\be \label{skoda_ineq}
% \nu_{z,0}(\varphi) \leq \nu_{z,n-1}(\varphi) \leq n \nu_{z,0}(\varphi).
%\ee
%\end{thm}
%In our context we have the following version of this result:
%\begin{lma} The following inequality holds:
%$$\frac{1}{n-t}\nu_{z,0}(\varphi) \leq \nu_{z,t}(\varphi) \leq \nu_{z,n-1}(\varphi).$$
%\end{lma}
%\begin{proof}
%Assume that $\nu_{z,0}(\varphi) = \tau > 0$. The inequality of Skoda implies that $\nu_{z}(\varphi) \geq \tau$
%which means that $$ e^{-2 \varphi(\zeta)} \geq \frac{1}{|\zeta -z |^{2  \tau}}.$$
%Thus $$ \nu_{z,t}(\varphi) \geq \inf \{s>0: \int_z \frac{1}{|\zeta -z |^{2 \tau/s+2t}}  d \lambda(\zeta) < + \infty \}=\frac{\tau}{n-t}.$$
%On the other hand 
%$$ \nu_{z,t}(\varphi) \leq \nu_{z,n-1}(\varphi),$$
%since $e^{-2t \log| \zeta -z|} \leq e^{-2(n-1) \log| \zeta -z|} $ for $0 \leq t \leq n-1$.
%This completes the proof.
%\end{proof}

We will now give a proof of the Skoda inequality (\ref{skoda_ineq}), based on the Oshawa-Takegoshi extension theorem, 
learned in a private communication
with Mattias Jonsson. We will also use the same technique to give a simple
proof of Theorem \ref{berndtssons_sats}.

We begin by recalling the statement of the \textbf{Oshawa-Takegoshi theorem} (see e.g. \cite{oh-t}):
Assume $V$ is a smooth hypersurface in $\mathbb{C}^m$ which in local coordinates
can be written as $ V = \{ z_n = 0 \}$, and let $U$ be a neighbourhood in $\mathbb{C}^m $ whose intersection 
with $V$ is non-empty. We also assume $\varphi$ is such that $\int_V e^{-2\varphi} < \infty$. Then, if $h_0 \in \mathcal{O}(V \cap U)$, there exists a $h \in \mathcal{O}( U)$
with $h=h_0$ on $V$ which satisfies the following estimate:
\be \label{ohtineq} 
\int_{U} \frac{ |h|^2 e^{-2 \varphi}}{|z_n|^{2-2 \delta}} \leq C_{ \delta} \int_{U \cap V} |h_0|^2 e^{-2 \varphi},
\ee
for $0 < \delta < 1$ and some constant $C_{\delta}$ depending only on $U$, $\delta$ and $\varphi$.  \\

The hard part of Skoda's inequality, and the part we will show, is the implication 
$\nu_{z,n-1}(\varphi)<1 \Rightarrow \nu_{z,0}(\varphi)<1.$ 
We record the core of the argument as a lemma.

\begin{lma} \label{jonssonlma}
Let $\varphi \in PSH(\Omega)$ and let $x \in \Omega$. Assume there exists a complex line $L$ through $x$
for which 
$$ \int_{L \cap \Omega} e^{-2 \varphi} < \infty,$$
then there exists a neighbourhood $\omega \subset \Omega$ of $x$, for which
$$  \int_{\omega} e^{-2 \varphi} < \infty.$$
\end{lma}

That is, in order to prove that $e^{-2 \varphi}$ is locally integrable at a point, we need only to find a complex line where the statement holds.

\begin{proof}
It suffices, of course, to prove this for $x=0$. 
Assume $L$ is a complex line through the origin for which $\int_{L \cap \Omega} e^{-2 \varphi|_L} < \infty$.
Applying the Ohsawa-Takegoshi extension theorem inductively, we can extend the function $1 \in L^2(L \cap \Omega, e^{- \varphi|_L}) \bigcap \mathcal{O} (L \cap \Omega) $ to a function
$ h \in L^2(\Omega,e^{- \varphi}) \bigcap \mathcal{O} (\Omega) $, where $\Omega$ is a neighbourhood in $\mathbb{C}^n$,
with a bound on the $L^2$ norm: 
$$ \int_{\Omega} |h|^2 e^{-2 \varphi } \leq C \int_{\Omega \bigcap L} e^{-2 \varphi |_L} < + \infty$$
for some constant $C$. This inequality is obtained from (\ref{ohtineq}) by just discarding the denominator figuring in
the left-hand-side integral.
Since $h$ is equal to 1 on $L$, the quantity $|h|^2$ is comparable
to 1 in a neighbourhood $\omega$ of the origin. Thus we obtain:
$$ \int_{\omega} e^{-2 \varphi } < + \infty$$
which is what we aimed for. 
\end{proof} 
\textit{Proof of Skoda's inequality:}

Remember, we want to prove the implication $\nu_{z,n-1}(\varphi)<1 \Rightarrow \nu_{z,0}(\varphi)<1$.
To that effect, assume $ \nu_{z,n-1}(\varphi) < 1.$ It is well known that
the Lelong number (at the origin) of a function $\varphi$ is equal to the Lelong number of the same function
restricted to a generic complex line passing through the origin (we will prove this later, see lemma \ref{linelemma}), which coincides with the integrability index
on that line. Thus we can find a complex line $L$ for which $\int_{\Omega \bigcap L} e^{-2 \varphi |_L} < + \infty$ and so by 
lemma (\ref{jonssonlma}) we see that $\nu_{z,0}(\varphi)<1 $. \\ \\

One might hope that knowledge of the dimension of the set where $\nu_{z,n-1}(\varphi) \geq c$ would enable
us to sharpen the estimate of Skoda's inequality. The following example shows that unfortunately this information is not sufficient to succeed.

\begin{ex}
Let $n=2$ and $\varphi(z_1,z_2)= \log( | z_1 |^2  + |z_2|^{2a})$. Then
one calculates:
\begin{itemize}
	\item $\nu_{0,n-1}( \varphi) = 2$,
	\item $\nu_{0,0}( \varphi) = \frac{2}{1+1/a}$,
	\item  $\{z: \nu_{z,n-1}( \varphi) \geq 1 \}=\{0 \}$.
\end{itemize}
This is the best scenario possible: the dimension of the upper-level set of the Lelong number is 0
and \textit{still} the lower bound of the Skoda inequality is sharp, which one realizes by letting $a \rightarrow \infty$.
\end{ex}

Let us see how we can apply the full strength of the estimate (\ref{ohtineq}) of the Ohsawa-Takegoshi theorem to obtain a proof of Theorem \ref{berndtssons_sats}. We recall the statement of Theorem \ref{berndtssons_sats}:
$$ \nu_{a} (\varphi) \geq 1 \Longleftrightarrow \int_a e^{-2 \varphi(\zeta) - 2(n-1)\log |\zeta - a|} d \lambda(\zeta) = + \infty.$$
Assume $a=0$, and let $\Omega$ be a small neighbourhood of the origin in $\mathbb{C}^n$, and let $\varphi \in PSH(\Omega)$ satisfy $\nu_0(\varphi)<1 $. Then we know that the restriction of $e^{-2\varphi}$ to a generic complex line is integrable. However, since a rotation of $\varphi$ will not effect $\varphi$'s integrability properties in $\Cn$, we may assume that the line is given by $\{z_2=...=z_n=0 \}$. In fact
we can assume that $\varphi$ is integrable along every coordinate axis. Thus $\varphi$ satisfies
\be \label{gen_cond}
\int_{ \{z_2=...=z_n=0 \} \cap \Omega} e^{-2 \varphi}d \lambda_1 < + \infty .
\ee
We want to prove that $$ \int_0 e^{-2 \varphi(\zeta) - 2(n-1)\log |\zeta|} d \lambda(\zeta) < +\infty.$$
If we consider the constant function 1 as an element of $\mathcal{O} (\mathbb{C} \cap \Omega )$ then, by the argument above, we obtain a function
$h \in \mathcal{O} (\mathbb{C}^2 \cap \Omega)$, comparable to 1 in $\Omega$, thus aquiring the following inequality:
$$ \int_{\mathbb{C}^2 \cap \Omega'} \frac{ e^{-2 \varphi}}{|z_1|^{2-2 \delta}} \leq C_{ \delta} \int_{\mathbb{C} \cap \Omega} e^{-2 \varphi}< + \infty,$$
with $\Omega' \subset \Omega$.
Since $0 < \delta < 1$ the function
$ \varphi + (1- \delta) \log |z_1| $ is plurisubharmonic in $\Omega$. Thus we can repeat the argument with $\varphi$
exchanged for $\varphi + (1- \delta) \log |z_1| $ to obtain: 
$$ \int_{\mathbb{C}^3 \cap \Omega''} \frac{ e^{-2 \varphi}}{|z_1|^{2-2 \delta}|z_2|^{2-2 \delta}} \leq C_{ \delta} \int_{\mathbb{C}^2 \cap \Omega} \frac{ e^{-2 \varphi}}{|z_1|^{2-2 \delta}} < + \infty,$$
with $\Omega'' \subset \Omega'. $
Iterating this procedure it is easy to realise that, after possibly shrinking $\Omega$, we obtain the inequality 
\be \label{skodaintegral}
\int_{ \Omega} \frac{ e^{-2 \varphi}}{|z_1|^{2-2 \delta}...|z_{n-1}|^{2-2 \delta}} \leq C_{ \delta} \int_{\mathbb{C}^{n-1} \cap \Omega} \frac{ e^{-2 \varphi}}{|z_1|^{2-2 \delta}...|z_{n-2}|^{2-2 \delta}} < + \infty.
\ee
Using the trivial estimate
$$ \int_{ \Omega} \frac{e^{-2 \varphi}}{|z|^{2(n-1)(1 -\delta)}} \leq \int_{ \Omega} \frac{ e^{-2 \varphi}}{|z_1|^{2-2 \delta}...|z_{n-1}|^{2-2 \delta}}$$
we see that
$$ \int_{ \Omega} e^{-2 \varphi -2(n-1)(1 -\delta)\log|z|} \leq C_{\delta}< + \infty, \, \, \, \, \, \, \forall \delta>0,$$ 
which, by lemma \ref{singularitylemma} with $\psi(\zeta)=(n-1) \log|\zeta|$, implies that  
$$ \int_{ \Omega} e^{-2 \frac{\varphi}{1+\delta(n-1)} -2(n-1)\log|z|} \leq C_{\delta}< + \infty, \, \, \, \, \, \, \forall \delta>0.$$ 
In this argument, since $\nu_0(\varphi)<1$, we can exchange $\varphi$ for $\frac{\varphi}{1-r} $, where $r>0$, and still have $\nu_0(\frac{\varphi}{1-r})<1$.
Thus, the hypothesis implies that $\nu_{0,n-1} (\varphi) < 1 $. \\The other direction is simpler: 
by introducing polar coordinates we see that  
$$ \int_a e^{-2 \varphi(\zeta) - 2(n-1)\log |\zeta - a|} d \lambda(\zeta) = C \int_{\omega \in S^{2n-2}} \int_{t \in \mathbb{C}, |t|<1} e^{-2 \varphi(a + t\omega)}.$$ If $ \nu_a(\varphi) \geq 1$ then the integral of $\varphi$ over almost every complex line through $a$ is infinite, and thus the above integral is infinite which is the same as saying $\nu_{0,n-1} (\varphi) \geq 1$.
We have proved Theorem \ref{berndtssons_sats}. \\ \\

We will now describe the relation between the generalized Lelong number $ \nu_{a,k}$ and restrictions to 
linear subspaces of dimension $k$. In order to do this, we will have to recall the natural measure
on the Grassmannian induced by the Haar measure on $U(n)$,
where $U(n)$ denotes the unitary group of $\mathbb{C}^n$. Also, let $\vartheta$ denote the unique, unit Haar measure
on $U(n)$. Then we can define a measure $d\mu$ on the Grassmannian $G(k,n)$ - the set of $k-$dimensional subspaces of
$\Cn$ - by setting for some fixed $V \in G(k,n)$, and $A \subset G(k,n)$  the mass of $A$ to be $\mu(A)= \vartheta ( M \in U(n) : MV \in A ).$ 
This means that if $ P : U(n) \rightarrow G(k,n)$ is the function $P(M) = MV$, then $$\mu = P_{*} (\vartheta) .$$
The measure $\mu$ is easily seen to be invariant under actions of $U(n)$, that is, $\mu(MA)= \mu(A)$ if $M \in U(n)$,
and also to be independent of our choice of $V$.
For $f$ a function defined on $G(k,n)$, we have
$$ 
\int_{T \in G(k,n)} f(T) d \mu = \int_{ M \in U(n)} f(MV) d \vartheta.
$$
We deduce that, for $g \in C_c^{\infty}(\Cn)$, 
$$ 
\int_{T \in G(k,n)} \int_{T} g(z) d \lambda_k d \mu =
\int_{ M \in U(n)} \int_{ MV } g(z) d \lambda_k d \vartheta 
$$
where $\lambda_k$ is the $k-$dimensional Lebesgue measure.
After changing to polar coordinates the above integral becomes 
$$ \int_0^{\infty}  \int_{M \in U(n)} \int_{ S_{MV}^{2k-1} } \rho^{2k-1} g( \rho \omega) dS(\omega) d \vartheta d \rho ,$$
where $S_{MV}^{2k-1}$ denotes the $(2k-1)-$dimensional sphere in the $k-$dimensional subspace defined by $MV$.
Consider the linear functional on $C(\rho S^{2n-1})$, defined by 
\be \label{functionaldef}
I_\rho(g) := \rho^{2n-1}\int_{M \in U(n)} \int_{S_{MV}^{2k-1} } g( \rho \omega ) dS(\omega) d \vartheta .
\ee
Notice that, even though this functional is defined for functions 
on $S^{2n-1}$ while the integration takes place on the sphere $S^{2k-1}$,
it is invariant under rotations on the sphere $S^{2n-1}$.
%This is so since we can get to any point on $S^{2n-1}$ by considering a point on some rotation of $S^{2k-1}$,
%and in (\ref{functionaldef}) we integrate over all rotations of the $(2k-1)$-dimensional sphere with respect to the uniform measure $d \mu $ ( i.e. $\mu(A) = \mu(MA), M \in U(n)$).

By the Riesz representation theorem, this functional is given by integration against a measure $d \gamma$ on $\rho S^{2n-1}$, i.e.,
\be \label{functionaldef2}
I_\rho(g)= \int_{\rho S^{2n-1}} g(\omega) d \gamma (\omega),
\ee
where, since $I_\rho$ is rotational invariant, the measure $d \gamma $ is rotational invariant as well. 
Thus $d \gamma$ is equal to the surface measure on
$\rho S^{2n-1}$ multiplied with a constant $c(\rho)$. This constant is easily determined
by evaluating $I_{\rho}(1)$ using the two expressions (\ref{functionaldef}),(\ref{functionaldef2}) above (remember that $\vartheta$ was normalized so that $\vartheta(U(n))=1$) :
$$\rho^{2n-1} c(\rho) = I_{\rho}(1) = \rho^{2n-1} \int_{S_{MV}^{2k-1} } dS(\omega). $$
So $c(\rho)=c_k=\int_{S^{2k-1} } dS(\omega)$ and is therefore independent of $\rho$.

Thus we see that the integral
$$ 
\int_{T \in G(k,n)} \int_{ T} g(z) d \lambda_k d \mu
=
\int_0^{\infty} \rho^{2(k-n)} I_\rho(g)   d \rho
$$
is equal to
$$
c_k \int_0^{\infty} \rho^{2(k-n)} \int_{\rho S^{2n-1}} g(\omega) d S (\omega) d \rho
=
c_k \int_{\Cn} |z|^{2(k-n)} g(z) d \lambda_n
$$

Exchanging $ g(z)$ for $ g(z) |z|^{2(n-k)} $ we have proven the following formula, which generalizes
the formula for changing to polar coordinates in an integral:
\begin{lma} \label{grassmanmeasurelemma}
For $g$ an integrable function,
$$\int_{\Cn} g d \lambda_n = c_k^{-1}\int_{T \in G(k,n)} \int_{ T} |z|^{2(n-k)} g(z) d \lambda_k d \mu .$$
\end{lma}
\begin{proof}
We have proven the formula under the condition that $g \in C_c^{\infty}(\Cn)$. The general case follows by approximating
an arbitrary integrable function $g$ by functions in $C_c^{\infty}(\Cn)$ .
\end{proof}

Of course, a similar formula holds if we instead consider $k-$planes through some arbitrary point $a \in \Cn$,
and in the above discussion assume the spheres to be centered around the point $a$.
This remark applies to the following result as well:

\begin{prop}
Let $k$ be an integer between 0 and $n-1$. Then $\nu_{0,n-k} (\varphi)<1$ 
$\Longleftrightarrow $
$ \nu_{0,0}(\varphi|_T)<1$ for almost every $T \in G(k,n)$
$\Longleftrightarrow$
$ \nu_{0,0}(\varphi|_T)<1$ for some $T \in G(k,n)$.
\end{prop}

\begin{proof}
The assumption $\nu_{0,n-k} (\varphi)<1$ means that 
$$\int_{B(0,r)} e^{ - 2 \frac{\varphi (\zeta)}{1 - \delta} - 2 (n-k) \log( \left| \zeta  \right| )  } d \lambda_n < +\infty$$
for some $r>0$, and $\delta>0$ small.
Using lemma \ref{grassmanmeasurelemma} this integral equals
$$ \int_{T \in G(k,n)} \int_{ B(0,r) \cap T} e^{ - 2 \frac{\varphi (\zeta)}{1 - \delta}} d \lambda_k d \mu.$$
Thus $\int_{ B(0,r) \cap T} e^{ - 2 \frac{\varphi (\zeta)}{1 - \delta}} d \lambda_k $ 
must be finite for almost every $T$ (since
by the lemma, $d \mu$ is a multiple of the Lebesgue measure),
which implies that $ \nu_{0,0}(\varphi|_T)<1$ for almost every $T \in G(k,n)$. This, of course,
implies that $ \nu_{0,0}(\varphi|_T)<1$ for some $T \in G(k,n)$.
On the other hand, 
if $\int_{ B(0,r') \cap T} e^{ - 2 \frac{\varphi (\zeta)}{1 - \delta}} d \lambda_k < +\infty$ for some $T$ and $\delta>0$,
then the exact same argument involved in proving the Skoda inequality 
(using the Ohsawa-Takegoshi theorem), shows that in fact
$$\int_{ B(0,r) } e^{ - 2 \frac{\varphi (\zeta)}{1 - \delta'} -2(n-k)\log|z|} d \lambda_n < +\infty,$$
for some small $\delta'>0$. This implies that $\nu_{0,n-k}(\varphi)<1.$
\end{proof}

A similar argument gives us the following classical statement (cf. \cite{Siu}):

\begin{thm}
For a generic $V \in G(k,n)$ we have that $$\nu_{0,n-1}(\varphi) = \nu_{0,k-1}(\varphi |_V).$$ 
\end{thm}

\begin{proof}
Assume $\nu_{0,n-1}(\varphi) < 1$. Then by Lemma \ref{grassmanmeasurelemma},
with $g(\zeta)=\exp(-2\varphi( \zeta) - 2(n-1)\log|\zeta| )$ we get that for some small $\delta>0$,
$$ 
+ \infty >
\int_{B(0,r)} e^{ - 2 \frac{\varphi (\zeta)}{1 - \delta} - 2 (n-1) \log \left| \zeta \right|   } d \lambda_n = 
\int_{V \in G(k,n)} \int_{ B(0,r) \cap V} e^{ - 2 \frac{\varphi (\zeta)}{1 - \delta} - 2 (k-1) \log \left| \zeta \right| } d \lambda_k d \mu. 
$$  
Thus $\nu_{0}(\varphi |_V) < 1$ for a generic $V \in G(k,n)$. The other direction is proved by 
using the same Ohsawa-Takegoshi argument as before.
\end{proof}

Taking $k=1$ we obtain again the classical result:

\begin{thm} \label{linelemma}
For almost every complex line $L$ through a point $a$, $$\nu_{a}(\varphi) = \nu_{a}(\varphi|_L).$$
\end{thm}

That is, the Lelong number coincides with what it generically is on complex lines. Moreover, $\nu_{n-k}$
coincides with the integrability index of $\varphi$ restricted a generic $k-$plane. \\ \\

\section{The analyticity property of the upper level sets of the generalized Lelong number} \label{analyticity_section}
In this section we prove that the upper level sets of our generalized Lelong number are analytic,
provided that the weight is "good enough". This we accomplish by
considering the Bergman function, whose definition we will soon recall. 
First, however, we begin with discussing what properties
the weight need to satisfy in order to be "good enough".\\ \\

$\bullet$ We say that a plurisubharmonic function $\psi$
has an isolated singularity at the origin 
if there exists an $M>0$ such that $$ \psi(z) \geq M \log|z|,$$
for $z$ close to $0$. 
It might be worth mentioning that in the case of analytic singularities, i.e., if
$\psi = \log|f|$ where $f=(f_1,..,f_n)$
is a tuple of holomorphic functions with common intersection locus at the a single point,
the least of all $M$ for which 
$$ \log|f(z)| \geq M \log|z|$$
is called the Lojasiewicz exponent of $f$. \\

$\bullet$ We assume as before that $$e^{-2(1+\tau)\psi} \in L_{Loc}^1(0),$$ for some $\tau>0$.\\

$\bullet$ We also assume that $e^{2 \psi}$ is Hölder continuous with exponent $\alpha$, at $0$. \\

$\bullet$ Finally we assume that $\nu_0(\psi) = l>0,$ so that $\psi$ carries some singularity at the origin.

\begin{df}
We say that a plurisubharmonic function $\psi$ is an admissible weight, 
and write $\psi \in W(\tau,l,M,\alpha)$ if it satisfies the four properties above.
\end{df} 

Admissible weights satisfy the following property, which we will make use of in the proof
of the analyticity:
%\begin{df}
%%We say that a plurisubharmonic function $\psi$ satisfies the annuli-property at a point $a$, and write
%$\psi \in M(a)$ if there exists a $R,C>0$, such that, 
%$$ e^{ -2 \psi(\zeta - a')} \geq Ce^{ -2 \psi(\zeta - a)}$$
%for $\zeta \in A(k)$ and $|a'|^=2^{-Rk}$. 
%\end{df}

\begin{lma} \label{annulilemma}
Assume that $e^{2\psi}$ is Hölder continuous at the origin, with exponent $\alpha$ 
and satisfies $\psi \geq M \log|z|$ near the origin.
Then there exists a $R>0$ and a constant $C>0$, such that for every $k\in \mathbb{N}$:
$$ 
e^{ -2 \psi(\zeta - a')} \geq C e^{ -2 \psi(\zeta)}
$$
where $ 2^{-(k-1)} \leq |\zeta | \leq 2^{-k}  $ and $|a'|^=2^{-Rk}$. 
\end{lma}

\begin{proof}
Fix $k\in \mathbb{N}$.
We want to show that $$  e^{ -2 \psi(\zeta-a')} \geq Ce^{ -2 \psi(\zeta)} $$
for $ 2^{-(k-1)} \leq |\zeta | \leq 2^{-k}  $ and $|a'|^=2^{-Rk}$. 
Since $|\zeta|^R \geq c |a'|$ the assumption gives us that 
$$e^{2 \psi(\zeta)} \geq |\zeta|^{2 M} \geq c |a'|^{ \alpha}$$ if $R \alpha \geq 2 M$ (which is a condition independent of $k$).
Now, using the Hölder continuity, we obtain 
$$ e^{ 2 \psi(\zeta - a')} \leq e^{ 2 \psi(\zeta)} + |a'|^{\alpha} \leq c e^{ 2 \psi(\zeta)}$$
which implies 
$$ e^{ -2 \psi(\zeta - a')} \geq C e^{ -2 \psi(\zeta)}, $$
for some contant $C>0$.

\end{proof}

%\begin{lma}
%If $\alpha_i >0$, then $\psi= \log(\sum z_i^{\alpha_i})$ satisfies the annuli-property.
%\end{lma}
%\begin{proof}
%We begin with the case $n=2$.
%Fix $k>0$, and take an arbitrary $z \in A(k)$, and $a$ such that $|a|=2^{-kR} $. 
%We want to show that $$ e^{ -2 \psi(\zeta - a)} \geq Ce^{ -2 \psi(\zeta - a)}$$
%which is equivalent to 
%$$|z_1 - a_1|^{\alpha_1} + |z_2 - a_2|^{\alpha_2} \leq C(|z_1|^{\alpha_1} + |z_2|^{\alpha_2}),$$
%where $C$ is some positive constant, independent of $k$, if $R$ is large enough.
%Assume, to begin with, that $|z_2|^2 \geq |z_1|^2$ so that $|z_2|^2 \geq 2^{-k-2}$. We assume
%this just to ensure that $|z_2|$ is far away from 0. In particular, we get that $|a_2|\leq| z_2|.$.
%Thus, $|z_2 - a_2|^{\alpha_2} \leq  (|z_2| + |a_2|)^{\alpha_2} \leq (2 |z_2|)^{\alpha_2},$ \\
%Assume furthermore, that $|z_1| \leq |a_1|.$ Then $|z_1 - a_1|^{\alpha_1} \leq (2 |a_1|)^{\alpha_1} \leq 2^{\alpha_1} 2^{-k R \alpha_1},$
%which is $\leq |z_2|^{\alpha_2} $ if $R$ is so large that $k R \alpha_1 \geq (k-2) \alpha_2 $. This condition on $R$ is satisfied if we choose $R>\frac{\alpha_2}{\alpha_1}$ which indepent of $k$. \\
%On the other hand, if $|z_1| \geq |a_1|.$ Then $|z_1 - a_1|^{\alpha_1} \leq (2|z_1|)^{\alpha_1}.$
%Thus $$|z_1 - a_1|^{\alpha_1} + |z_2 - a_2|^{\alpha_2} \leq C(|z_1|^{\alpha_1} + |z_2|^{\alpha_2})$$
%if $R$ is large enough.
%The case $|z_2|^2 \geq |z_1|^2$ follows the exact same reasoning.
%For $n>2$, a simple induction argument gives the result.
%\end{proof}

\begin{ex} 
Examples of plurisubharmonic functions which are admissible weights are given by
$$ \psi = \log\left( \sum_{i=1}^n |f_i|^{\alpha_i} \right)$$
where $f=(f_1,...,f_n)$ is an $n-$tuple of holomorphic functions with common zero locus at the origin.
Here we have to assume that $\alpha_i>0$ are as small as needed in order for a $\tau>0$ to exist for which 
$$e^{-2(1+\tau)\psi} \in L_{Loc}^1(0).$$
Then $e^\psi$ is Hölder continuous with Hölder exponent $\min(1,\alpha_i)$, and $\psi$ have Lelong number equal to
$\min{\alpha_i \nu_0(\log|f_i|)}.$ 
Also, the Lojasiewicz exponent, which is the smallest $M$ for which $ \sum_{i=1}^n |f_i|^{\alpha_i} \geq |z|^M$ 
is easily seen to be finite.
\end{ex}

We now define the Bergman kernel with respect to a weight.
\begin{df}
Let $a \in \Omega$, $\varphi \in PSH(\Omega)$ and $\psi \in W(\tau,l,M, \alpha)$. We define
$\mathcal{H}_a = \mathcal{O}(\Omega) \cap L^2 ( \Omega, e^{-2\varphi( \cdot) - 2 \psi( \cdot - a )} ),$
which is a separable Hilbert space. The Bergman kernel for a point $z \in \Omega$ is defined
as the unique function $B^{\psi}_a(\zeta,z)$, holomorphic in $\zeta$, satisfying
$$ 
h(z) = \int_{\Omega} h( \zeta ) \overline{B^{\psi}_a(\zeta,z)} e^{-2\varphi( \zeta) - 2 \psi( \zeta - a )} d \lambda(\zeta),
$$
for every $h \in \mathcal{H}_a$.
\end{df}
The existence of the Bergman kernel is a (rather easy) consequence of the Riesz representation theorem for Hilbert spaces.
Closely related to the Bergman kernel is the Bergman function:

\begin{df}
For $a\in \Omega$ the Bergman function at a point $\zeta \in \Omega$ is defined as 
$$ B^{\psi}_a(\zeta) :=  B^{\psi}_a(\zeta,\zeta).$$ 
\end{df}
We define
$$\Lambda(a) = \{ h \in \mathcal{O}(\Omega) , \int_{\Omega} \left| h( \zeta ) \right| ^2 
e^{ - 2 \varphi (\zeta) - 2\psi( \zeta-a  )} d \lambda ( 	\zeta) \leq 1  \}$$
that is, those functions in $\mathcal{H}_a$ of norm less than or equal to 1.
Let us calculate the norm of $\zeta \mapsto B^{\psi}_a(\zeta,z)$:
$$
\left\| B^{\psi}_a(\cdot,z) \right\|^2 =  
\int_{\Omega} B^{\psi}_a(\zeta,z) \overline{B^{\psi}_a(\zeta,z)} e^{-2\varphi( \zeta) - 2 \psi( \zeta - a )} d \lambda(\zeta)
= B^{\psi}_a(z,z),
$$
which in particular implies that $ B^{\psi}_a(z,z)$ is given by a non-negative real number.
Consequently $$s(\zeta) = \frac{B^{\psi}_a(\zeta,z)}{ \sqrt{B^{\psi}_a(z,z)}} \in \Lambda(z)$$ and 
so 
\be \label{bmanrealization}
|s(z)|^2 = B^{\psi}_a(z).
\ee
\\
Also, we have that
\be
\sup_{h \in \Lambda(a)} |h(z)|^2 = \sup_{h \in \Lambda(a)} |(h,B^{\psi}_a(\cdot,z))|^2 = \left\| B^{\psi}_a(\cdot,z) \right\|^2 =
 B^{\psi}_a(z),
\ee
where $(,)$ denotes the inner product in $\mathcal{H}_a$,
which gives us the following extremely useful characterization of the Bergman function:
\be \label{bmanchar}
B_a^{\psi}(z) := \sup \{ |h(z)|^2 : h \in \mathcal{O}(\Omega) , \int_{\Omega} \left| h( \zeta ) \right| ^2 
e^{ - 2 \varphi (\zeta) - 2\psi( \zeta-a  )} d \lambda ( 	\zeta) \leq 1  \} ,
\ee
and (\ref{bmanrealization}) means that this supremum is actually realized by the function $s$. \\
%\begin{df} \label{bergmanfncdef}
%For $z, a \in \Omega$, $\varphi \in PSH(\Omega)$ and $\psi \in W(\tau,l,M)$, 
%the Bergman function of $\Omega$ with respect to the weight 
%$${ \zeta \mapsto -\varphi (\zeta) -  \psi( \zeta-a)  } $$ 
%is defined as 
%$$B_a^{\psi}(z) := \sup \{ |h(z)|^2 : h \in \mathcal{O}(\Omega) , \int_{\Omega} \left| h( \zeta ) \right| ^2 
%e^{ - 2 \varphi (\zeta) - 2\psi( \zeta-a  )} d \lambda ( 	\zeta) \leq 1  \} .
%$$
%We denote set which we take the supremum over by $\Lambda(a)$.
%\end{df}

Bergman functions enjoy several nice properties, and one, critical for our purposes, is provided by the following
theorem of Berndtsson (cf. \cite{Berndtsson}), 

\begin{thm}
If $\Omega$ is pseudoconvex, then the function $(a,z) \mapsto \log B_a^{\psi}(z)$, is plurisubharmonic in $(a,z)$.
\end{thm} 

Thus we can talk about the Lelong number of the function $ z \mapsto \log B_z^{\psi}(z)$ in $\Omega$,
and the following proposition relates it to the generalized Lelong number of $\varphi$.
More specifically, it says that if the generalized Lelong number of $\varphi$ with respect to $\psi$ is larger than 1, then the classical Lelong number
of $ z \mapsto \log B_z^{\psi}(z)$ is larger than 0, and if the generalized Lelong number is smaller than 1 the
classical number is 0.  In the terminology of Kiselman, we say that $\log B_z^{\psi}(z)$ attenuates the singularities of $\varphi$.
\\ \\
Recall that by lemma (\ref{annulilemma}) we can find a $R>0$ 
for which
$$ e^{ -2 \psi(\zeta-a')} \geq C e^{ -2 \psi(\zeta)}$$
for $ 2^{-(k-1)} \leq |\zeta | \leq 2^{-k}  $ and $|a'|^=2^{-Rk}$, for every $k \in \mathbb{N}$. \\ \\
Also, by lemma (\ref{singularitylemma}) we can choose an $\epsilon < \delta \tau$ (arbitrarily close to $\delta \tau$)  for which
 \be \label{epsilonn}
 \int_{B(a,1/2^N)} e^{ - 2 \varphi (\zeta) - 2(1 -   \epsilon ) \psi( \zeta-a)  }= \infty,
 \ee
if we fix an $N>0$ large enough.
\begin{prop} \label{kernellemma}
Let $\delta>0$ be small and let $\Omega$ be an open and pseudoconvex set containing the point $a$,
and let $\psi \in W(\tau,l,M,\alpha)$. Assume $$\nu_{a,\psi} (\varphi)=1+\delta.$$
Then, with $\cdelta = \delta \tau l $, 
the classical Lelong number of $\log  B_{\cdot}^{\psi}(\cdot)$ at
$a$ is larger than or equal to $\cdelta/R$ , that is:
$$ \nu_a(\log B_{\cdot}^{\psi}(\cdot)) \geq \frac{\cdelta}{R}. $$
On the other hand, if we assume that
$$\nu_{a,\psi} (\varphi)<1,$$
then
$$ \nu_a(\log B_{\cdot}^{\psi}(\cdot)) = 0. $$
\end{prop}

Without loss of generality, we can assume that $a=0$. 
Recalling the definition of the classical Lelong number, we see that we want to show that 
\be \label{lelongdef} 
\lim_{r \rightarrow 0}  \frac{ \sup_{|z|=r} \log B_{z}^{\psi}(z) }{\log r}  \geq C_{\delta} /R .
\ee
The idea of the proof is the following. 
The assumption $\nu_{0,\psi} (\varphi)=1+\delta$ essentially means that $$\int_0 e^{-2 \frac{\varphi}{1+ \delta} - 2\psi} = \infty.$$
Thus, the weight $\varphi + \psi$ has a rather large singularity at the origin. If we move the singularity of $\psi$ by translating it
to an arbitrary point $a'$,
then if $a'$ is small enough, the weight $\varphi(\zeta) + \psi(\zeta - a')$ will have a rather large singularity at the point $a'$.
As we will show, the singularity will actually be so large that if $h \in \Lambda(a')$, that is if $h$ is holomorphic and satisfies
$$
\int_{\Omega} \left| h( \zeta ) \right| ^2 
e^{ - 2 \varphi (\zeta) - \psi( \zeta-a'  )} d \lambda ( 	\zeta) \leq 1  ,
$$
then $h$ is forced to be small at the point $a'$. In fact, 
\be \label{smallness}
|h(a')|^2 \leq |a'|^{C_\delta/ R} .
\ee
This would be  enough to prove the proposition, but the following observations show that in fact it will be enough to show something weaker.
By Cauchy estimates we will see that it suffices to find just some point $z_0$ near the origin for which $|h(z_0)|^2 \leq |z_0|^{C_\delta }$,
if $h \in \Lambda(a')$. 
%Thus we need only show that a function $h \in \Lambda(a')$ 
%is small in the sense that $ |h(z_0)|^2 \leq |z_0|^{C_\delta } $ at some point $z_0$ close enough to the origin, and not just %at the point $a'$. 
This simplifies things considerably.
Also, since we, in (\ref{lelongdef}), are dealing with a limit, it suffices to find a sequence $r_k$ tending to 0, for which the inequality (\ref{lelongdef}) holds.
This means that instead applying the above idea to arbitrary points $a'$ near the origin, we merely need to apply it to points of a sequence $a_k$
tending to 0. We now turn to the details. 
\\
\\

\begin{lma} \label{sequencelemma}
Assume $\nu_{0,\psi} (\varphi)=1+\delta$.   
Fix any sequence $a_k \rightarrow 0$ with $|a_k|=2^{-Rk}$ and for every $k$ choose a corresponding $h^k \in \Lambda(a_k)$.
Then $\{ a_k \}$ contains asubsequence $\{ a_{k_j} \}$, for which there exists $b_{k_j} \in B(0,2^{-{k_j}})$ with, 
$$|h^{k_j}( b_{k_j} )| \leq  \left|b_{k_j} \right| ^{C_{\delta} }.$$ 
\end{lma}
\begin{proof}
The lemma is proven under the assumption $C_\delta = \epsilon l$. 
The general case with $C_\delta = \delta \tau l$ then follows, since $\epsilon$ can be choosen
arbitrarily close to $\delta \tau$.
We will prove the lemma by arguing via contradiction. The negation of the statement is the following:
%\begin{center}
%1) $a_k \rightarrow 0$ with $|a_k|=2^{-Rk}$ and \\
%2) $h^k \in \Lambda (a_k)$ satisfying, for $k$ large enough,
For every $k$ larger than some finite number, which we can assume to be the $N$ figuring in ($\ref{epsilonn}$),
$$|h^k( \zeta )| >  \left|\zeta \right| ^{C_{\delta} },$$
for every $\zeta \in B(0,2^{-k})$.
%\end{center}
Let us assume this negation.
Then, for every $a_k$ we have, since $h^k \in \Lambda(a_k) $, 
$$
1 \geq
\int_{B(0, 2^{-k})}   \left|   \zeta  \right| ^{{C_{\delta}}}e^{ - 2 \varphi (\zeta) - 2 \psi( \zeta -a_k)}d \lambda ( 	\zeta) \geq
\int_{A(k)}   \left|   \zeta  \right| ^{{C_{\delta}}}e^{ - 2 \varphi (\zeta) - 2 \psi( \zeta -a_k)}d \lambda ( 	\zeta),
$$
where $A(k)$ denotes the annulus $B(0,2^{-k}) \setminus B(0,2^{-k-1})$.   
Since for $\zeta \in A(k)$
$$ e^{ -2 \psi(\zeta-a_k)} \geq C e^{ -2 \psi(\zeta)} $$
by lemma \ref{annulilemma}, we deduce that
\be \label{annulusintegral}
C \geq
 \int_{A(k)}   \left|   \zeta  \right| ^{{C_{\delta}}}e^{ - 2 \varphi (\zeta) - 2 \psi( \zeta)}d \lambda ( 	\zeta).
\ee
Now, remember that $\epsilon$ was choosen so that
$$ \infty = \int_{B(0,1/2^N)} e^{ - 2 \varphi (\zeta) - 2(1 -   \epsilon ) \psi( \zeta)  }d \lambda( \zeta).$$  
Thus, by covering the ball $B(0,1/2^N)$ by annuli $B(0,2^{-k}) \setminus B(0,2^{-k-1})$, we get
\begin{eqnarray*}
\infty = \int_{B(0,1/{2^N})} e^{ - 2 \varphi (\zeta) - 2(1 -   \epsilon ) \psi( \zeta)  }d \lambda( \zeta)
\leq
 C \sum_{k}
 \int_{A(k)}   \left| \zeta  \right|^{2 \epsilon l} e^{ - 2 \varphi (\zeta) -  2 \psi( \zeta) } d \lambda ( 	\zeta) \\
\leq 
 C \sum_k  2^{-k \epsilon l} 
 \int_{A(k)}    \left| \zeta \right| ^{\epsilon l} e^{ - 2 \varphi (\zeta) - 2 \psi( \zeta )}d \lambda ( 	\zeta)
 \leq (\ref{annulusintegral}) \leq 
  C \sum_k 2^{-k \epsilon l}  < + \infty,
 \end{eqnarray*}
 where we in the first inequality use the fact that $\nu_0(\psi)<l \Rightarrow e^{2 \epsilon \psi( \zeta)} \leq C | \zeta|^{2 \epsilon l} $ for $|\zeta|\leq2^{-N}$, if $N$ is large enough.
 This establishes the desired contradiction.
 \end{proof}

 \begin{proof} \textit{(of Proposition (\ref{kernellemma}))}
Fix a point $a \in \Omega$ with $|a|=2^{-Rk}$ for some $k\in \mathbb{N}$, and choose an $h \in \Lambda(a),$
for which $|h( b )| \leq  \left| b \right| ^{C_{\delta}}$ for some $b \in B(0,2^{-{k}})$.
We claim that in fact such an $h$  satisfies the following estimate:
\be \label{contprop}
|h(a)| \leq D |a|^{C_{\delta}/R}, 
\ee 
for some constant $D\geq0$ which does not depend on $h$ nor $k$:
Since $ \varphi$ and $\psi$ are locally bounded from above, every $h\in\Lambda(a)$ satisfies:
$$ \int_{ \Omega } \left| h( \zeta ) \right| ^2 d \lambda ( \zeta) \leq C .$$ Thus, for $\Omega' \subset \subset \Omega$,
by applying Cauchy estimates on $h$, we see that  $$\left| h'( \zeta ) \right| ^2 \leq C ,$$ in $ \Omega'$, where $C$ is some constant independent of ${k}$. 
By using a first order Taylor expansion of $h$, we conclude that 
$$|h(a)| \leq |h(b)|+C|a-b| \leq |b|^{C_\delta} + C|b|      \leq D \left|a \right|^{C_{\delta}/R},$$ 
for some constant $D\geq0$ independent of $h$ and $k$, as promised. \\

Let us complete the proof:
Assume that $\nu_{0, \psi}(\varphi) = 1 + \delta.$ 
Choose $\rho>0$ and fix a sequence $\{ a_k \} \subset \Omega$ with $|a_k|=2^{-Rk}$ satisfying the following:
$$ \sup_{|z|=2^{-Rk}} \log B_{z}^{\psi}(z) \leq \log B_{a_k}^{\psi}(a_k) + \rho,$$
for every $k$. 
In view of (\ref{bmanrealization}),
which says that we can actually find a holomorphic function in $ \Lambda(z) $ realising
the Bergman function at $z$,
 we can find for each $k$, a funcion $h^k \in \Lambda(a_k)$ for which 
\be \label{specreal}
h^k(a_k) = B_{a_k}^{\psi}(a_k) .
\ee
By lemma (\ref{sequencelemma}) we can extract a subsequence $\{a_{k_j} \}$ with a corresponding sequence $\{ b_{k_j} \} $, where 
$|b_{k_j}| =2^{-k_j} $,
for which $$ |h^{k_j}(b_{k_j})| \leq |b_{k_j}|^{C_{\delta}} .$$
The estimate (\ref{contprop}) implies that $$ \log B_{a_{k_j}}^{\psi}(a_{k_j}) \leq \log| a_{k_j}|^{\cdelta / R} + D.$$
%We have thus shown the following fact, which is the key to proving the lemma: 
%Given any sequence $ \{ a_k \}$ with $|a_k|=2^{-Rk}$, we can extract a subsequence for which 
%$$ \log B_{a_{k_j}}^{\psi}(a_{k_j}) \leq \log |a_{k_j}|^{\cdelta} + D.$$

%Then, by the above, we can extract the subsequence of $\{ a_k \}$ for which 
%$$ \log B_{a_{k_j}}^t(a_{k_j}) \leq \log \left| a_{k_j} \right| ^{C_{\delta}} + D.$$
Thus we obtain (observe that the denominator is \textit{negative}),
$$ \lim_{r \rightarrow 0}  \frac{ \sup_{|z|=r} \log B_{z}^{\psi}(z) }{\log r} 
= 
\lim_{j \rightarrow \infty }  \frac{ \sup_{|z|=2^{-R k_j}} \log B_{z}^{\psi}(z) }{\log 2^{-R k_j}} 
\geq 
\lim_{j \rightarrow \infty }  \frac{  \log B_{a_{k_j}}^{\psi}(a_{k_j}) + \rho}{\log 2^{-R k_j}} 
\geq 
\cdelta / R,
$$
and we are done in this case.\\ \\
If $\nu_a(\varphi,\psi)<1$, then
an application of Hörmanders $L^2$-methods (cf. \cite{Hormander})  provides us with an holomorphic function $h$, satisfying $|h(a)|^2 > 0$, and the integral over $\Omega$ of $h$ with respect to the weight $ e^{ - 2 \varphi (\cdot) - 2 \psi( \cdot - a )} $ is less than $1$. In view of (\ref{bmanchar}) this implies that $B_a^{\psi}(a) > 0$, hence $ \nu_a(\varphi, \psi)=0$.

\end{proof}

We can now prove the analogue of Siu's theorem for our generalized Lelong number, using an argument due to Kiselman. 
\begin{thm} \label{analyticity_thm}
Let $\Omega \in \mathbb{C}^n$ be open and pseudoconvex and $\varphi$ be a plurisubharmonic function in $\Omega$.
Then if $\rho > 0$ ,
$$ \{ z \in \Omega : \nu_{z,\psi}( \varphi) \geq \rho  \}$$
is an analytic subset of $\Omega$.
\end{thm}

\begin{proof}
We first note that if $\psi=0$ then $\nu_{a,0}( \varphi)$ is the same as the integrability index of $\varphi$
for which the conclusion of the theorem holds (see e.g. \cite{Kiselman}). 
Using the notation of proposition \ref{kernellemma} we define  $$ \Psi(z) = 3 n \frac{\log B_z^{\psi} (z)}{\cdelta / R}, z \in \Omega.$$
The core of the proof is to show that
$$ \{z \in \Omega:  \nu_{z,\psi}( \varphi) \geq 1 + \delta\}
\subset
\{z \in \Omega:  e^{-2\Psi} \notin L_{Loc}^1(z) \}
\subset
\{z \in \Omega:  \nu_{z,\psi}( \varphi) \geq 1\}.
$$
This we can accomplish, as follows: \\ \\
If for $a \in \Omega$ we have that $\nu_{a,\psi}( \varphi) \geq 1 + \delta $ then due to Proposition \ref{kernellemma}, the classical Lelong number of $\Psi$ at $a$ is greater than $3n$ since 
$$\nu_a( \Psi) \geq \frac{3n \cdot \cdelta/R}{\cdelta /R} = 3 n.$$
By Skoda's inequality ($\ref{skoda_ineq}$) we have that $\nu_a( \Psi) \leq n \nu_{a,0} (\Psi)$ which shows that the integrabilty index of $\Psi$ at $a$ is larger than or equal to $3$. In particular, this
implies that $e^{-2 \Psi( \cdot ) }$ is not locally integrable at $a$ and proves the first of the inclusions. \\ 

For the second one, assume that $$ \nu_{a,\psi}(\varphi) < 1.$$
This implies that $ e^{ -2 \varphi (\zeta) - 2\psi ( \zeta - a )} $ is locally integrable at $a$. As noted above, an 
application of Hörmanders $L^2$-methods gives us a holomorphic function $h$ in $\Omega$ such that $|h(a)|^2 > 0$, and the integral of $h$ with respect to the weight $ e^{ - 2 \varphi (\cdot) - 2 \psi( \cdot - a )} $ is less than $1$.
Thus the function $ z \mapsto B_z^{\psi} (z)$, being defined as that supremum of the modulus square of all holomorphic functions whose integral with respect to our weight is less than or equal to 1, 
is strictly positive at $a$, which implies $$ \Psi(a)>-\infty.$$ 
But (see e.g.\cite{Hormander}) every plurisubharmonic
function $u$ satisfies that $e^{-2u}$ is locally integrable around the points where it is finite, and thus we see that
$$ e^{-2 \Psi} \in L_{Loc}^1(a),$$
which proves the second inclusion.\\
\\
As noted above, we know that set $\{z \in \Omega:  e^{-2\Psi} \notin L_{Loc}^1(z) \}$ is analytic in $\Omega$. Thus, by rescaling
we obtain analytic sets $Z_{\delta,\rho}$ such that 
$$ \{z \in \Omega :  \nu_{z,\psi} ( \varphi) \geq \rho \}
\subset
Z_{\delta,\rho}
\subset
\{z \in \Omega:  \nu_{z,\psi}( \varphi) \geq \frac{\rho}{1+\delta}\},
$$
which implies
$$ \{z \in \Omega:  \nu_{z,\psi}( \varphi) \geq \rho \} = \bigcap_{\delta>0} Z_{\delta,\rho}.$$
Since the intersection of any number of analytic sets is analytic, we are done.
\end{proof}

As a conequence of this theorem we can define the following concept, introduced in the classical
case by Siu(\cite{Siu}):

\begin{df}
For $Z$ an analytic set in $\Omega$, we define the \textit{generic generalized Lelong number} of $\varphi$ by  $$ m_Z^{\psi}( \varphi) = \inf \{\nu_{z,{\psi}}( \varphi); z \in Z  \}$$
\end{df}
We have the following lemma by precisely the same argument as in the classical case.
\begin{lma}
$\nu_{z,{\psi}} (\varphi) = m_Z^{\psi}( \varphi)$ for $z \in Z \setminus Z'$ where $Z'$ is a union of countably many proper analytic subsets of $Z$. 
\end{lma}

\begin{proof}
Put $Z' = \bigcup_{c>m_Z^{\psi}, c \in \mathbb{Q}} Z \cap E_c^{\psi}$, where $E_c^{\psi} = \{ z \in Z : \nu_{z,{\psi}} (\varphi) \geq c \}. $
Then each $Z \cap E_c^{\psi}$ is an analytic proper subset of $Z$ and $\nu_{z,{\psi}} (\varphi) = m_Z^{\psi}( \varphi)$ on $Z \setminus Z'$
by construction.
\end{proof}

\section{Approximation of plurisubharmonic functions by Bergman kernels} \label{approx_section}
A well known result due to Demailly (see for instance \cite{DemaillyKollar}) makes it possible to approximate a plurisubharmonic
function $\varphi$, by the logarithm of the Bergman function ($\Psi^m$) with respect to the weight $e^{-2m\varphi}$, as $m$ tends to infinity. 
%The last condition means that each 
%$\Psi_i$ is of the type $\log ( \sum |g_i| )$ where $g_i \in \O (\Omega).$  
Furthermore, the approximation is continuous with respect to the (classical) Lelong number: $$\nu_{z,n-1}(\Psi^m) \rightarrow \nu_{z,n-1}(\varphi),$$
as $m \rightarrow \infty$.
We will now show that the same holds true using the Bergman function with respect to the weight $e^{-2m\varphi-2\psi(\cdot - x)} $
where $x$ is the point at which we evaluate the Bergman function. 
The argument mimics closely that of Demailly's (cf. \cite{DemaillyKollar}), with some minor changes to fit our case.
To begin with, we modify the construction of $\mathcal{H}_a$ slightly:
\begin{df}
For each $m\in \mathbb{N}$ and $a \in \Omega$ we let 
$$\mathcal{H}_a^m = \mathcal{O}(\Omega) \cap L^2 ( \Omega, e^{-2 m \varphi( \cdot) - 2 \psi( \cdot - a )} ).$$
\end{df}
Denote by $B_{a}^m$ (for notational convinience we supress the dependence on $\psi$) the Bergman function for $\mathcal{H}_a^m$, and
put $$ \Psi_a^m(z) = \frac{1}{2 m} \log{B_{a}^m}(z),$$
for $z\in \Omega$.
Fix $a \in \Omega$. If $h \in \H_a^m$ and has norm bounded by 1, 
the mean value property for holomorphic functions shows that for $r=r(a)>0$ small enough, 
\begin{eqnarray} 
|h(a)|^2 
&\leq&
\frac{n!}{\pi^n r^{2n}} \int_{|a-\zeta|<r} |h(\zeta)|^2 d \lambda ( \zeta)  \leq  \nonumber \\
&\leq& \frac{n!}{\pi^n r^{2n}} e^{\sup_{|a-\zeta|<r} \{ 2 m  \varphi( \zeta) + 2 \psi(\zeta-a) \} }   
\int_{|a-\zeta|<r} |h(\zeta)|^2 e^{- 2 m \varphi( \zeta) - 2 \psi(\zeta-a)} d \lambda ( \zeta). \nonumber 
\end{eqnarray}

Thus, if we assume that 
\be \label{psihypotes2}
\psi( \zeta -a ) \leq l \log|\zeta-a|,
\ee
we have that
$$\Psi_a^m(a) 
\leq 
\sup_{|a-\zeta|<r} \{  \varphi( \zeta) +  \frac{1}{2m} 2 \psi(\zeta-a) \} - \frac{1}{2m} \log r^{2n} + C
\leq
\sup_{|a-\zeta|<r} \{  \varphi( \zeta) \} +\frac{1}{m}(l-n) \log r + C/m.$$

Now, assume that 
\be \label{psihypotes} 
\psi(\zeta) \geq (n-\delta)\log|\zeta|,
\ee
for some small, fixed $\delta > 0$.
Fix a point $a$ for which $\varphi(a)>- \infty$.
By considering the 0-dimensional variety $\{ a \}$, we obtain, by the Ohsawa-Takegoshi theorem (see section \ref{properties_and__examples}), the following: 
for every $\xi \in \C$, there exists an $h \in \O(\Omega)$,
satisfying $h(a)=\xi$, and $C_{\delta}>0$ a constant, depending only on the dimension and $\delta$, such that 
$$ \int_{ \Omega} |h(\zeta)|^2 e^{- 2 m \varphi( \zeta) - 2 (n-\delta)\log| \zeta-a|} d \lambda ( \zeta) 
\leq 
C_{\delta} e^{-2m \varphi(a) } |\xi|^2.
$$ 
By the assumption (\ref{psihypotes}) this implies that
$$ \int_{ \Omega} |h(\zeta)|^2 e^{- 2 m \varphi( \zeta) - 2 \psi(\zeta-a)} d \lambda ( \zeta) 
\leq 
C_{\delta} e^{-2m \varphi(a) } |\xi|^2.
$$ 
Since this holds for every $\xi$ we can choose one such that the right-hand-side is equal to 1, i.e. $C_{\delta} e^{-2m \varphi(a) } |\xi|^2=1$. 
Then $h$ satisfies $$ \int_{ \Omega} |h(\zeta)|^2 e^{- 2 m \varphi( \zeta) - 2 \psi(\zeta-a)} d \lambda ( \zeta) 
\leq 
1
$$ 
and 
$$ 
\log|h(a)|^2 = \log|\xi|^2 = -\log C_{\delta} + 2 m \varphi(a).
$$
Thus, 
$$ \Psi_a^m(a) \geq \varphi(a) - \frac{\log C_{\delta}}{2m}.$$
If $a$ is such that $\varphi(a)=- \infty$ this inequality is trivial.
%************************************************************************
%************************************************************************
Thus, for every $m$ and $z \in \Omega$ we have that 
\be \label{lelongapprox} 
\varphi(z) - \frac{1}{2m} \log C_{\delta} \leq \Psi_z^m(z) \leq  \sup_{|a-\zeta|<r} \{  \varphi( \zeta) \} +\frac{1}{m}(l-n) \log r + C. 
\ee
We now want to show that this approximation behaves well with respect to generalized Lelong number with weight $\psi$.
To this end, fix $a\in \Omega$ and let $\lambda > \nu_{a,\psi}( \Psi_{ a}^m( \cdot) )$, and put
$$ p = 1 +m \lambda, q=1+\frac{1}{m \lambda}.$$
Then $1/p +1/q=1$  and we apply Hölder's inequality to the following integral, with $r(a)$ so small that $ \{ |a-\zeta|<r(a) \} \subset \subset \Omega$,
%$$ 
%\int_{|\zeta - a|<r/N} e^{-\frac{2m}{p} \varphi - 2 \psi(\zeta-a) } d \lambda(\zeta)
%=
%\int_{|\zeta - a|<r/N} e^{-\frac{2m}{p} \varphi - 2 \psi(\zeta-a)} (B_{m,\zeta}(\zeta))^{1/p}(B_{m,\zeta}(\zeta))^{-1/p} d \lambda(\zeta),
%$$
$$ 
\int_{|\zeta - a|<r} e^{-\frac{2m}{p} \varphi - 2 \psi(\zeta-a) } d \lambda(\zeta)
=
\int_{|\zeta - a|<r} e^{-\frac{2m}{p} \varphi - 2 \psi(\zeta-a)} (B_{a}^m(\zeta))^{1/p}(B_{a}^m(\zeta))^{-1/p} d \lambda(\zeta),
$$
to obtain, since $-q/p = -1/(m \lambda)$, that it is dominated by 
$$ 
\Big( \int_{|\zeta - a|<r} (B_{a}^m(\zeta)) e^{-2m \varphi}  e^{- 2 \psi(\zeta-a)} d \lambda(\zeta) \Big)^{1/p}
\Big(\int_{|\zeta - a|<r} (B_{a}^m(\zeta))^{-1/m \lambda} e^{ - 2 \psi(\zeta-a)}d \lambda(\zeta)\Big)^{1/q}. 
$$
%By (\ref{bergmandomination}), the first integral is finite, and the second is finite by assumption on $\lambda$.
The first integral can be seen to be finite (cf. \cite{DemaillyKollar},p.21), as well as the second integral, since it equals 
$$ 
\int_{|\zeta - a|<r}  e^{- 2 \log B_a^m(\zeta)/{(2 m \lambda)} - 2 \psi(\zeta-a) } d \lambda(\zeta), 
$$
which is finite due to the assumption on $\lambda$.
Since $\frac{1}{m/p} = \frac{1}{m} + \lambda$  this implies the inequality:
\be \label{psiestimate}
\nu_{a,\psi}(\varphi(\cdot)) \leq \nu_{a,\psi}(\Psi_{a}^m(\cdot)) + 1/m.\ee
We need a lemma:

\begin{lma}
For every $m \in \mathbb{N}$ and $a \in   \Omega$ we have that 
$$ B_{a}^m(\zeta) \geq  C |\zeta-a|^{(\frac{n-\delta}{l})(n+2)} \cdot B_{\zeta}^m(\zeta)$$ for every $\zeta$ in a small neighbourhood of $a$, where $C>0$ is an constant not depending on $\zeta$ or $a$.
\end{lma}

\begin{proof}
For a fixed point $a$ in $\Omega$, choose $\zeta \neq a$ with distance less than $\min \{1,dist(a,\partial \Omega) \}$ to each other , and let $$M=\frac{n-\delta}{l}.$$ First, we claim that for $$z \in B(\zeta,2^{-M} |\zeta-a|^{M}),$$ the following inequality holds:
\begin{equation} \label{weightestimate}
e^{-2\psi(a-z)}\leq e^{-2\psi(\zeta-z)}.
\end{equation} 
Indeed, thanks to the assumptions (\ref{psihypotes2}) and (\ref{psihypotes}), for such $z$,
$$ e^{\psi(\zeta-z)}\leq |\zeta-z|^l \leq |a-z|^{n-\delta} \leq  e^{\psi(a-z)},$$  
since $$ |\zeta-z|\leq 2^{(-M)} |\zeta-a|^{M} \leq |z-a|^{M},$$
where in the second inequality we used that $|\zeta-a|\leq 2|z-a| $.
Now, let \[
h\in\Lambda^m(\zeta):=\{h\in\mathcal{O}(\Omega),\int_{\Omega}\left|h(z)\right|^{2}e^{-2m\varphi(z)-2\psi(z-\zeta)}d\lambda(z)\leq1\}\]
 be such that $h(\zeta)=B^m_{\zeta}(\zeta)$. To simplify notation we assume $m=1$, but the following calculations remains valid for any $m$. 
 Take a smooth function $\theta$ with
support in the ball $B(\zeta,2^{-M}|\zeta-a|^{M})$ satisfying $\theta=1$
in a neighbourhood of $B(\zeta,2^{-(1+M)}|\zeta-a|^{M})$, and \begin{equation}
|\bar{\partial}\theta(z)|\leq\frac{{1}}{|a-\zeta|^{2M}}.\label{eq:differentialineq}\end{equation}
Thus for every point $z\in supp\theta$ the inequality (\ref{weightestimate}) holds.
Moreover, we have that \[
e^{-2(n+1)\log|z-\zeta|}=\frac{{1}}{|z-\zeta|^{2(n+1)}}\leq \frac{2^{(n+1)(M+1)}}{|a-\zeta|^{2M(n+1)}}\]
for $z\in supp\bar{\partial}\theta$ .
Putting this information together we obtain the following estimate:

\begin{equation} \label{integralestimate}
\int_{\Omega}|\bar{\partial}\theta h|^{2}e^{-2\varphi(z)-2\psi(a-z)-2(n+1)\log|z-\zeta|}d\lambda(z)
\leq \end{equation}
\[
\leq\frac{2^{(n+1)(M+1)}}{|a-\zeta|^{2M(n+1)}}
\int_{\Omega}|\bar{\partial}\theta h|^{2}e^{-2\varphi(z)-2\psi(\zeta-z)}d\lambda(z)
\leq
\frac{2^{(n+1)(M+1)}}{|a-\zeta|^{2M(n+2)}},\]
since $h\in\Lambda(\zeta).$ Thus, by standard $L^2$-estimates, we can solve the equation \begin{equation}
\bar{\partial}u=\bar{\partial}(\theta h)=\bar{\partial}\theta\cdot h\label{eq:dbar}\end{equation}
with respect to the weight $e^{-2\varphi(z)-2\psi(a-z)-2(n+1)\log|z-\zeta|}.$
The singularity of the weight forces $u$ to vanish at $\zeta$, so if
we define $F=\theta h-u,$ then $F(\zeta)=h(\zeta),$ and $F$ is holomorphic
in $\Omega.$ Moreover, by the triangle inequality, \[
\left(\int_{\Omega}|F|^{2}e^{-2\varphi(z)-2\psi(a-z)}d\lambda(z)\right)^{1/2}\leq\left(\int_{\Omega}|\theta h|^{2}e^{-2\varphi(z)-2\psi(a-z)}d\lambda(z)\right)^{1/2}+\]
 \[
+\left(\int_{\Omega}|u|^{2}e^{-2\varphi(z)-2\psi(a-z)}d\lambda(z)\right)^{1/2}.\]
 Using (\ref{weightestimate}) we have that \[
\int_{\Omega}|\theta h|^{2}e^{-2\varphi(z)-2\psi(a-z)}d\lambda(z)\leq\int_{\Omega}|\theta h|^{2}e^{-2\varphi(z)-2\psi(\zeta-z)}d\lambda(z)\leq1,\]
and also we see that \[
\int_{\Omega}|u|^{2}e^{-2\varphi(z)-2\psi(a-z)}d\lambda(z)\leq \int_{\Omega}|u|^{2}e^{-2\varphi(z)-2\psi(a-z)-2(n+1)\log|z-\zeta|}d\lambda(z)\leq\]
\[
\leq C' \int_{\Omega}|\bar{\partial}\theta h|^{2}e^{-2\varphi(z)-2\psi(a-z)-2(n+1)\log|z-\zeta|}d\lambda(z),\]
where the first inequality comes from the assumption that $|z-\zeta|\leq|a-\zeta|\leq 1$, and the last inequality, as well as the constant $C'$ (which only depends on $\Omega$), comes from the $L^{2}-$estimate obtained
from solving \eqref{eq:dbar}. 
%However, by the assumptions of $\theta$
%we have that \[
%e^{-2(n+1)\log|z-\zeta|}=\frac{{1}}{|z-\zeta|^{2(n+1)}}\leq4^{2(n+1)}\frac{{1}}{|a-\zeta|^{2M(n+1)}}\]
%for $z\in supp\bar{\partial}\theta$ , so the last integral is bounded
%by \[
%\frac{{C}}{|a-\zeta|^{M2(n+1)}}\int_{\Omega}|\bar{\partial}\theta h|^{2}e^{-2\varphi(z)-2\psi(\zeta-z)}d\lambda(z)\leq\]
%\[
%\leq\frac{{C}}{|a-\zeta|^{M2(n+1)+2M}}\int_{\Omega}|h|^{2}e^{-2\varphi(z)-2\psi(\zeta-z)}d\lambda(z)\leq C\frac{{1}}{|a-\zeta|^{2M(n+1)+2M}},\]
%(where we once again used \eqref{eq:assumption}) which 
Using (\ref{integralestimate}) we arrive at
\[
\int_{\Omega}|F|^{2}e^{-2\varphi(z)-2\psi(a-z)}d\lambda(z)\leq1+C'\frac{2^{(n+1)(M+1)}}{|a-\zeta|^{2M(n+2)}}\leq\frac{C_1}{|a-\zeta|^{2M(n+2)}},\]
where $C_1$ is a constant independent of $\zeta$ and $a.$ Thus, if
we define the function $$\tilde{F}(z)=\frac{|a-\zeta|^{M(n+2)}}{\sqrt{C_1}}F(z)$$
then $\tilde{F}$ belongs to $\Lambda(a)$ and satifies \[
|\tilde{F(\zeta)|}=C_1^{-1/2}|B^{m}_{\zeta}(\zeta)|\cdot|a-\zeta|^{M(n+2)}.\]
This shows that for every $a$ and $\zeta$ (the inequality is trivial
if $\zeta=a$), \[
|B^{m}_{a}(\zeta)|\geq C |B^m_{\zeta}(\zeta)|\cdot|a-\zeta|^{M(n+2)},\]
where $C=C_1^{-1/2}$ does not depend on $\zeta$ or $a$. 
\end{proof}

Using the lemma we see that for each $a \in \Omega$ ,
$$\frac{{1}}{2m}\log|B^{m}_{a}(\zeta)|\geq\frac{{1}}{2m}\log|B^{m}_{\zeta}(\zeta)|+\frac{{M(n+2)}}{2m}\log|a-\zeta|+C/m,$$
for $\zeta$ close to $a$,
which implies that\[
\nu_{a,\psi}(\Psi_a^m(\cdot))\leq\nu_{a,\psi}(\Psi_{\cdot}^m(\cdot))+\frac{{C}}{2m}.\]
Combining with (\ref{psiestimate}) we obtain that
$$ \nu_{a,\psi}(\varphi(\cdot)) \leq \nu_{a,\psi}(\Psi_{\cdot}^m(\cdot)) + C/m,$$
for every $a \in \Omega$.

On the other hand, the left-hand estimate of (\ref{lelongapprox}) implies that
$$ \nu_{a,\psi}(\varphi(\cdot)) \geq \nu_{a,\psi}(\Psi_{\cdot}^m(\cdot)).$$ Thus we have proved:
%************************************************************************
%************************************************************************
\begin{thm}
Assume $\psi$ satisfies
\be 
l\log|z| \geq \psi(z) \geq (n-\delta)\log|z|,
\ee
for some small, fixed $\delta > 0$.
Then for $\varphi \in PSH(\Omega)$, $m \in \mathbb{N}$, $z,a \in \Omega$ and every $r<d(z, \ds \Omega)$ we have that 
\be \label{lel_appr}
\varphi(z) - \frac{1}{2m} \log C_{\delta} \leq \Psi_z^m(z) \leq  \sup_{|z-\zeta|<r} \{  \varphi( \zeta) \} +\frac{1}{m}(l-n) \log r + C/m
\ee
and
\be \label{lel_appr_label}
\nu_{a,\psi}(\varphi(\cdot)) - \frac{C}{m} \leq \nu_{a,\psi}(\Psi_{\cdot}^m (\cdot)) \leq \nu_{a,\psi}(\varphi(\cdot)),
\ee
where $C$ is a constant depending on $\Omega,l,\delta$. 
In particular, $\Psi_{z}^m( z )$ converges to $\varphi(z)$ as $m \rightarrow + \infty$ both pointwise and in $L_{Loc}^1$.
\end{thm}

\begin{remark}
If $\psi=0$, as in the original theorem of Demailly, then
$\Psi_a^m$ are plurisubharmonic functions with analytic singularities, and thus we approximate $\varphi$ with plurisubharmonic functions with analytic singluarities. For instance,
this gives a very simple proof of Siu's analyticity theorem for the classical Lelong number. In our setting however, it is unclear, and an interesting question, if the presence of $\psi$ allows for $\Psi_{a}^m $ to have analytic singularities.
\end{remark}

\begin{remark}
One can show that when comparing the approximations  $\Psi_z^m$ to the classical Lelong number we can obtain, instead of
(\ref{lel_appr_label}), the following inequalities:
$$ \nu_{a,n-1}(\varphi(\cdot)) - \frac{n-l}{m} \leq \nu_{a,n-1}(\Psi_{\cdot}^m (\cdot)) \leq \nu_{a,n-1}(\varphi(\cdot)).
$$
The approximation of Demailly, that is with $\psi=0$, satisfied these inequalities with $l=0$.
\end{remark}

\section{Kiselman's directed Lelong number} \label{kiselmans_section}

In this section we relate the classical Lelong number to yet another integral. Using this
relation, we can interpret Kiselman's directional Lelong number (defined below) as our generalized Lelong number, in a generic sense.
Equation (\ref{skodaintegral}) shows us that if $ \nu_0 (\varphi)<1$ and the integral of $e^{-2 \varphi}$  restricted to the coordinate axes is finite, then 
$$ \int_0 e^{ \frac{2\varphi}{1-\epsilon} -2(1-\delta) \sum_{i \neq j} \log|z_i| } d \lambda(z) < +\infty,$$
for every $ \delta >0$ and $j=1...n$, and for some small $\epsilon>0$.
Applying the inequality between geometric and arithmetic mean (i.e. $ \frac{a_1+...+a_n}{n} \geq (a_1 \cdots a_n)^{\frac{1}{n}}$
 for $ a_i \geq 0 $) we see that $$\frac{1}{n}\sum_{j=1}^n e^{-2 \sum_{i \neq j} \log|z_i| } = \frac{1}{n}\sum_{j=1}^n \frac{1}{\prod_{i \neq j} |z_i|^2} \geq \prod_{j=1}^n \frac{1}{\prod_{i \neq j} |z_i|^{\frac{2}{n}}} = \prod_{i=1}^n \frac{1}{|z_i|^{\frac{2(n-1)}{n}}}.$$
This last expression is equal to
$$ e^{-2 \frac{n-1}{n}\sum_{i = 1}^n \log|z_i| }$$
which tells us that
$$ \int_0 e^{ - \frac{2\varphi}{1-\epsilon} -2(1-\delta)\frac{n-1}{n}\sum_{i = 1}^n \log|z_i|} d \lambda(z) < + \infty.$$
Using a similar Hölder argument as in Lemma \ref{singularitylemma}, we get
 $$ \int_0 e^{ - \frac{2\varphi}{1-\epsilon+\delta'} -2\frac{n-1}{n}\sum_{i = 1}^n \log|z_i|} d \lambda(z) < + \infty,$$
for $\delta'>0$ sufficiently small. Furthermore, if $\delta'$ is so small that $\delta' < \epsilon$, then 
 $$ \int_0 e^{ - 2\varphi -2\frac{n-1}{n}\sum_{i = 1}^n \log|z_i|} d \lambda(z) < + \infty.$$
On the other hand, if 
$$ \int_0 e^{ - 2\varphi -2\frac{n-1}{n}\sum_{i = 1}^n \log|z_i|} d \lambda(z) < + \infty$$
then, using the pointwise estimate $ \log|z_i| \leq \log|z|$, we see that 
$$ \int_0 e^{ - 2\varphi -2(n-1)\log|z|} d \lambda(z) < + \infty$$
which is equivalent to $ \nu_{0,n-1} (\varphi) <1.$
Thus we have proved the following lemma:
\begin{lma} \label{dir_lel_lemma}
For a plurisubharmonic function $\varphi$, which satisfies that the restriction of $e^{-2 \varphi}$ to the coordinate axes is integrable,  the condition $\nu_{0,n-1} (\varphi) <1$ is equivalent to
$$ \int_0 e^{ - 2\varphi -2\frac{n-1}{n}\sum_{i = 1}^n \log|z_i|} d \lambda(z) < + \infty.$$
\end{lma}
\begin{remark}
Instead of saying that the lemma holds under the condition that the integral of the restriction of $e^{-2 \varphi}$ to the coordinate axes is finite, we can say that it holds for some generic rotation of $e^{-2 \varphi}$.
This is so since if the Lelong number of $\varphi$ is smaller than 1, then a generic rotation of $e^{-2 \varphi}$  is integrable along the coordinate axes.
\end{remark}
We now define the directed Lelong number, due to Kiselman (cf. \cite{Kiselman}).
\begin{df}
For $a_j \geq 0$ the directed Lelong number at a point $w$ is defined as 
$$ \nu_{w} (\varphi,(a_1,...a_n)) = \limsup_{r \rightarrow 0} \frac{ \sup_{|z_i-w_i|=r^{a_i} } \varphi(z_1,..,z_n) }{\log r}.$$
\end{df} 
It is proved in \cite{Kiselman} that the function which the limsup is taken over is increasing, and
so we can exchange the limsup for a limit. Also, for $a_i=1$ for every $i$, we obtain the classical Lelong number.
\begin{lma}
For $a_i = \frac{p_i}{q} \in \mathbb{Q}_+$ the directional Lelong number satisfies
$$\nu_{w} (\varphi,(a_1,...a_n)) = q^{-1}\nu_{w} (\varphi(z_1^{p_1},...,z_n^{p_n})).$$
\end{lma}
\begin{proof}
Since the result is local we can assume that $w=0$. By homogenity of the directional Lelong number it is enough
to consider the case $q=1$.
We have, with $z_i=r_i e^{i \theta_i}$, that
\be \label{dir_lel_expr} 
\sup_{|z_i|=r^{p_i} } \varphi(z_1,..,z_n) = \sup_{r_i=r^{p_i} } \varphi(r_1 e^{i \theta_1},..,r_n e^{i \theta_n}) =
\sup_{r_i=r } \varphi(r_1^{p_i} e^{i \theta_1},..,r_n^{p_n} e^{i \theta_n}).
\ee
Since $p_j \in \mathbb{N}_+$ this last expression is equal to 
$$ \sup_{r_i=r } \varphi(r_1^{p_i} e^{i \theta_1 p_1},..,r_n^{p_n} e^{i \theta_n p_n}) = 
\sup_{|z_i|=r } \varphi(z_1^{p_1} ,..,z_n^{p_n}),
$$
and we are done.
\end{proof}

\begin{prop} \label{thm_gen_dir_lelong_numbers}
Let $\varphi$ be plurisubharmonic in $\Omega$. Then, for a generic rotation of $\varphi$,
$$\nu_{w} (\varphi,(a_1,...a_n))<1$$ 
iff 
$$\int_{w} e^{-\frac{2\varphi}{q} - \sum_{i=1}^n (1-\frac{1}{n p_i})\log|z_i-w_i| } d \lambda(z) < +\infty,$$
for $a_i = \frac{p_i}{q} \in \mathbb{Q}_+$. 
\end{prop}

\begin{proof}
Assume $w=0$.
By the above lemma the hypotesis implies,
$$ \nu_{w} (q^{-1} \varphi(z_1^{p_1},...,z_n^{p_n})) < 1 $$
which by Lemma \ref{dir_lel_lemma} implies that 
\be \label{sumintergal}
\int_0 e^{ - 2q^{-1} \varphi( z_1^{p_1},...,z_n^{p_n}  ) -2 \frac{n-1}{n} \sum_{j = 1}^n \log|z_j|} d \lambda(z) < + \infty.
\ee
By applying the change of coordinates $z_j^{p_j} = y_j$, we see that the integral in (\ref{sumintergal}) is equal to
$$\int_0 e^{ - 2 q^{-1}  \varphi(y_1,...,y_n) -2 \sum_{j = 1}^n \log|y_j|({1- \frac{1}{p_j} + \frac{n-1}{n} \frac{1}{p_j}})} d \lambda(y) $$ 
Since $1- \frac{1}{p_j} + \frac{n-1}{n} \frac{1}{p_j} = 1 - \frac{1}{np_j}$, this equals
$$
\int_0 e^{ - 2q^{-1}\varphi(y_1,...,y_n) -2  \sum_{j = 1}^n \log|y_j|({1- \frac{1}{n p_j}})} d \lambda(y),
$$
which consequently is finite. This proves the "only if" part. However, the exact same argument used in the opposite direction
proves the "if" part.
\end{proof}

\begin{remark}
It is easy to see that $\varphi$ must satisfy some condition in terms of integrability.
Take for instance $\varphi = \log{z_1}$ for which $\int_{ \{ z_2=...=z_n=0\}} e^{-2 \varphi} d \lambda(z) = + \infty$. Then, it is easily verified that $\nu_{w} (\varphi,(1,...1)) = 1$ but of course,
$$\int_{w} e^{-2\varphi - \sum_{i=1}^n (1-\frac{1}{n})\log|z_i-w_i| } d \lambda(z) = +\infty.$$
\end{remark}
One knows that for the weight $ \psi = \log \max_i {|z_i|^{1/a_i} }$
the relative type, $\sigma(\varphi, \psi)$ and Demailly's generalized Lelong number $\nu_{Demailly}(\varphi, \psi)$ (defined in the introduction) both coincide with Kiselman's directed Lelong number.
However, it is unkown if there exists a weight $\psi$ for which $\nu_{z,\psi}$ coincides
with the directional Lelong number in more than the generic sense found above.
%it is unknown if this is also the case for Berndtsson's generalized Lelong number.
\def\listing#1#2#3{{\sc #1}:\ {\it #2}, \ #3.}

\end{document}